%% file: cutHHO_rev.tex
\newif \ifSIAM \SIAMfalse

\ifSIAM
\documentclass[10pt]{siamltex}
\else
\documentclass[a4paper,11pt]{article}
\usepackage{geometry}
\fi
\usepackage{epsfig}
\usepackage{amssymb}
\usepackage{amsbsy}
\usepackage{amsmath}
\usepackage{enumerate}
\usepackage{float}
\usepackage{graphicx}
\usepackage{xspace}
\usepackage{stmaryrd}

\usepackage{bm}
\usepackage{color}
\newcommand{\cor}[1]{#1}

\usepackage{remarks}
\ifSIAM
\newtheorem{assumption}[theorem]{Assumption}
\else
\usepackage{amsthm}
\newtheorem{theorem}{Theorem}
\newtheorem{lemma}[theorem]{Lemma}
\newtheorem{corollary}[theorem]{Corollary}
\newtheorem{assumption}[theorem]{Assumption}
\fi
\usepackage{def} 
 
\usepackage[numbers]{natbib}

\begin{document}

\newcommand\footnotemarkfromtitle[1]{%
\renewcommand{\thefootnote}{\fnsymbol{footnote}}%
\footnotemark[#1]%
\renewcommand{\thefootnote}{\arabic{footnote}}}

\title{An unfitted Hybrid High-Order method for elliptic interface problems\footnotemark[1]}
\author{Erik Burman\footnotemark[2] \and Alexandre Ern\footnotemark[3]}

\ifSIAM
\date{Draft version \today}
\fi

\maketitle

\renewcommand{\thefootnote}{\fnsymbol{footnote}} \footnotetext[1]{Draft
  version, \today}
\footnotetext[2]{Department of Mathematics, University College London, London, 
UK--WC1E  6BT, UK.}
\footnotetext[3]{Universit\'e Paris-Est, CERMICS (ENPC),
  77455 Marne-la-Vall\'ee cedex 2, and INRIA, Paris, France.}

\renewcommand{\thefootnote}{\arabic{footnote}}

\begin{abstract} 
We design and analyze a Hybrid High-Order (HHO) method on unfitted meshes to approximate elliptic interface problems. The curved interface can cut through the mesh cells in a very general fashion. As in classical HHO methods, the present unfitted method introduces cell and face unknowns in uncut cells, but doubles the unknowns in the cut cells and on the cut faces. The main difference with classical HHO methods is that a Nitsche-type formulation is used to devise the local reconstruction operator. As in classical HHO methods, cell unknowns can be eliminated locally leading to a global problem coupling only the face unknowns by means of a compact stencil. We prove stability estimates and optimal error estimates in the $H^1$-norm. Robustness with respect to cuts is achieved by a local cell-agglomeration procedure taking full advantage of the fact that HHO methods support polyhedral meshes. Robustness with respect to the contrast in the material properties from both sides of the interface is achieved by using material-dependent weights in Nitsche's formulation.
\end{abstract}

\ifSIAM

\fi

\pagestyle{myheadings} \thispagestyle{plain} 
\ifSIAM
\markboth{E. BURMAN, A. ERN}{An unfitted HHO method}
\fi

\section{Introduction} 

The Hybrid High-Order (HHO) method has been recently introduced
in~\cite{DiPEr:15} for linear elasticity problems
and in~\cite{DiPEL:14} for diffusion problems.
The HHO method is formulated in terms of cell 
and face unknowns. The cell unknowns can be eliminated locally 
by using a Schur complement technique (also known as static condensation),
leading to a global transmission problem coupling only the face
unknowns by means of a compact stencil. 
The HHO method is devised locally from two ingredients: 
a reconstruction operator and a stabilization operator. 
This leads to a discretization method that 
supports general meshes (with possible polyhedral cells and non-matching
interfaces), is locally conservative 
and delivers energy-norm error estimates of 
order $(k+1)$ (and $L^2$-norm error estimates of order $(k+2)$ under full
elliptic regularity) if polynomials of order $k\ge0$ are used for the 
face unknowns. As shown in~\cite{CoDPE:16}, the HHO method
can be fitted into the family of Hybridizable Discontinuous Galerkin (HDG) 
methods introduced in~\cite{CoGoL:09} and is closely related to the nonconforming Virtual Element Method (ncVEM) studied in \cite{AyLiM:16}.

The use of polyhedral meshes can greatly simplify the meshing of complicated geometries. Nevertheless, in some situations, it is still convenient to avoid the meshing of boundaries and internal interfaces. This is the case when the boundary changes during the computation, such as in free-boundary and optimization problems, and when the boundary or the internal interface is curved. In this paper, we are interested in devising a high-order approximation method for elliptic interface problems. To handle difficulties with curved interfaces in classical finite element methods, boundary-penalty methods \cite{Babuska:70,BE87} have been proposed, where the computational mesh does not need to respect the interface. In order to improve the accuracy, unfitted finite element methods were introduced in \cite{HH02} drawing on the seminal ideas of Nitsche \cite{Nit71} for the weak imposition of boundary conditions. The key idea is to design the finite element space so that singularities over the interface can be represented by a pair of polynomials in the cut cells. Similar approaches were then proposed in the context of discontinuous Galerkin methods in~\cite{BasEn:09,Massjung12,JL13}. 

A well-known difficulty for unfitted finite element methods is that the conditioning of the resulting linear system has a strong dependence on how the interface cuts the mesh cells. This means that for unfavorable cuts, Nitsche's formulation can be severely ill-conditioned. This difficulty has been solved in \cite{HH02} by using weighted coupling terms with cut-dependent weights. However, there is a lack of robustness when the material properties (e.g., the diffusivities on each side of the interface) are highly contrasted. Robustness with respect to the contrast can be achieved by using material-dependent weights, as proposed in different contexts in \cite{BZ06,ESZ09,BGSS17}, and in this case, a different mechanism is needed to handle unfavorable cuts. In the case of $H^1$-conforming methods, this problem can be overcome by adding a penalty term that weakly couples the polynomial approximation in adjacent cells as proposed in~\cite{Burman:10}. When using a discontinuous Galerkin approximation, another approach was proposed in~\cite{JL13} for fictitious domain problems where mesh cells with unfavorable cuts are merged with neighboring elements having a favorable cut. This idea is also explored in \cite{HWX17} for interface problems approximated by conforming finite elements on quadrilateral meshes whereby cells with an unfavorable cut are merged with adjacent quadrilateral cells (thus creating hanging nodes). 

The so-called cutFEM framework was developed recently in~\cite{BCHLM15} so as to couple different physical models over unfitted interfaces and to discretize PDEs over unfitted embedded submanifolds. The high-order approximation of the geometry of the interface was considered recently in~\cite{BHL16} using a boundary correction based on local Taylor expansions and in~\cite{LR16} using an iso-parametric technique, the common objective being to simplify the numerical integration on domains with curved boundaries by allowing a piecewise affine representation of the interface. The cutFEM paradigm has also been applied to a variety of complex flow problems, see, e.g., \cite{MaScW:18}, the recent PhD thesis \cite{Schott17}, and references therein. A conforming finite element method with local remeshing in subcells, effectively fitting the mesh to the interface, followed by elimination of the local degrees of freedom, was introduced in \cite{FR14}.

The goal of the present work is to devise and analyze an HHO method using unfitted meshes. The approach consists of doubling the unknowns in the cut cells and the cut faces, in a spirit similar to unfitted finite element methods. For brevity, we only consider elliptic interface problems, but the material can be readily adapted to treat the (simpler) case of fictitious domain problems; such an adaptation is briefly reported in \cite{BE:17}. Our approach combines the ideas of HHO methods (and more broadly HDG methods) with those from \cite{HH02} concerning Nitsche's formulation, but with material-dependent weights rather than cut-dependent weights, and those from~\cite{JL13} to handle unfavorable cuts by a local cell-agglomeration procedure. The cell-agglomeration procedure takes full advantage of the fact that the HHO method supports general meshes with polyhedral cells. The resulting unfitted HHO method is robust with respect to the cuts and to the material properties. Our stability and error analysis of the unfitted HHO method sheds some novel light in the analysis of HHO methods. On the one hand, the local reconstruction operator is based on Nitsche's formulation and cannot be related, as in classical HHO methods, to a local elliptic projector. On the other hand, the error is measured by using some projector that is somewhat more elaborate than the local $L^2$-orthogonal projector used in classical HHO methods. Our main result is an $H^1$-error estimate of order $(k+1)$ if polynomials of order $k\ge0$ are used for the face unknowns and polynomials of order $(k+1)$ are used for the cell unknowns. We observe that we do not consider here cell unknowns of order $k$ as in classical HHO methods. The overhead induced by this modification is marginal since, as usual, all the cell unknowns can be eliminated locally. \cor{Finally, we mention the recent numerical work combining the HDG method with the X-FEM technique for fictitious domain \cite{GurSKF:16} and elliptic interface \cite{GurKF:17} problems. The main differences with the present unfitted HHO method is that we do not introduce unknowns at the interface (but rather double the unknowns at the mesh faces cut by the interface) and that we provide a thorough analysis including robustness with respect to cuts and contrast.}

This paper is organized as follows. In Section~\ref{sec:model}, we introduce the elliptic interface problem we want to approximate. 
In Section~\ref{sec:discrete}, we present the discrete setting, including our
main notation for the cut cells and the two assumptions we require on the mesh, and 
we prove two key trace inequalities under these assumptions. 
In Section~\ref{sec:HHO}, we present the unfitted HHO method.
In Section~\ref{sec:proof}, we present our stability and error analysis;
our main result is Theorem~\ref{th:main}.
Finally, in Section~\ref{sec:build_mesh} we show how the two mesh properties
introduced in Section~\ref{sec:discrete} can be satisfied 
by using a local cell-agglomeration
procedure (under the assumption that the mesh is
fine enough to resolve the interface). 
Computational results will be reported in a separate publication.

\section{Model problem} 
\label{sec:model}

Let $\Omega$ be a domain in $\Real^d$ (open, bounded, connected, Lipschitz subset) and consider a partition of $\Omega$ into two disjoint subdomains so that $\overline\Omega = \overline{\Omega^1} \cup \overline{\Omega^2}$ with the interface $\Gamma = \partial \Omega^1 \cap \partial \Omega^2$. The unit normal vector $\bn_\Gamma$ to $\Gamma$ conventionally points from $\Omega^1$ to $\Omega^2$. For a smooth enough function defined on $\Omega$, we define its jump across $\Gamma$ as $\jump{v}_\Gamma := v_{|\Omega^1} - v_{|\Omega^2}$. We consider the following interface problem:
\begin{subequations}\label{eq:interface}\begin{alignat}{2}
- \DIV ( \kappa \GRAD u) &= f, &\qquad&\text{in $\Omega^1\cup\Omega^2$}, \\
\jump{u}_\Gamma &= g_{\textup{D}}, &\qquad&\text{on $\Gamma$}, \\
\jump{\kappa\GRAD u}_\Gamma\SCAL\bn_\Gamma & = g_{\textup{N}}, &\qquad&\text{on $\Gamma$}, \\
u&=0 &\qquad&\text{on $\partial\Omega$},
\end{alignat}\end{subequations}
with $f\in\Ldeux$, $g_{\textup{D}}\in H^{\frac12}(\Gamma)$, $g_{\textup{N}}\in L^2(\Gamma)$. For simplicity we consider a homogeneous Dirichlet condition on $\partial\Omega$. To avoid technicalities, we assume that the diffusion coefficient $\kappa$ is scalar-valued and that $\kappa^i:=\kappa_{|\Omega^i}$ is constant for \cor{each} $i\in\{1,2\}$. Without loss of generality, we assume that the numbering of the two subdomains is such that $\kappa^1 < \kappa^2$. In the rest of the paper, we assume that the interface $\Gamma$ is a smooth $(d-1)$-dimensional manifold of class $C^2$ that is not self-intersecting. This assumption can be relaxed at the price of additional technical issues that are not explored herein. 

\section{Discrete setting}
\label{sec:discrete}

We assume that the domain $\Omega$ is a polyhedron with planar faces in $\Real^d$. Let $\famTh$ be a shape-regular family of matching meshes covering $\Omega$ exactly. The meshes can have cells that are polyhedra with planar faces in $\Real^d$, and hanging nodes are also possible. The mesh cells are considered to be open subsets of $\Real^d$. For a subset $S\subset \Real^d$, $h_S$ denotes the diameter of $S$, and for a mesh $\calT_h$, the index $h$ refers to the maximal diameter of the mesh cells. The shape-regularity criterion for polyhedral meshes is that they admit a matching simplicial sub-mesh that satisfies the usual shape-regularity criterion in the sense of Ciarlet and such that each sub-cell (resp., sub-face) belongs to only one mesh cell (resp., at most one mesh face). The shape-regularity of the mesh sequence is quantified by a parameter $\rho\in(0,1)$ (see Section~\ref{sec:build_mesh} for further insight). In what follows, $B(\by,a)$ denotes the open ball with center $\by$ and radius $a$\cor{, $d(\by,A)$ denotes the distance of the point $\by$ to the set $A$, and $d(A,A')$ denotes the Hausdorff distance between the two sets $A,A'$.}

\subsection{Main notation for unfitted meshes}

Since the meshes are not fitted to the subsets $\Omega^1$ and $\Omega^2$, there are mesh cells in $\calT_h$ that are cut by the interface $\Gamma$. Let us define the partition $\calT_h = \calT_h^1\cup \calT_h^\Gamma\cup \calT_h^2$, where the subsets
\begin{subequations}
\begin{align}
\calT_h^i&:= \{ T\in \calT_h \tq T \subset \Omega^i\}, \quad \forall \iot,
\\
\calT_h^\Gamma &:= \{ T\in \calT_h \tq \mes_{d-1}(T\cap\Gamma)>0\},
\end{align} 
\end{subequations}
collect, respectively, the mesh cells inside the subdomain $\Omega^i$, $\iot$, and the mesh cells cut by the interface $\Gamma$. For any mesh cell $T\in \calT_h^\Gamma$ cut by the interface, we define 
\begin{equation}
T^i:=T\cap\Omega^i, \qquad
T^\Gamma:=T\cap\Gamma.
\end{equation} 
The boundary of the sub-cell $T^i$ is decomposed as follows:
\begin{equation}
\partial T^i=(\partial T)^i \cup T^\Gamma,
\end{equation} 
with the notation $(\partial T)^i = \partial T\cap\Omega^i$. For any mesh cell $T\in\calT_h$, the set $\calF_{\dT}$ collects the mesh faces located at the boundary $\dT$ of $T$. Whenever $T\in\calT_h^\Gamma$, we consider the set 
\begin{equation} \label{eq:def_calFTi}
\calF_{(\dT)^i} = \{ F^i=F\cap \Omega^i\tq F\in\calF_{\dT},\, \mes_{d-1}(F^i)>0\}.
\end{equation}
The sub-faces in $\calF_{(\dT)^i}$ form a partition of $(\partial T)^i$ (but not of $\partial T^i$ since $T^\Gamma$ is not included in $\calF_{(\dT)^i}$). The notation is illustrated in Figure~\ref{fig:cutcell1}. Since the interface $\Gamma$ is not self-intersecting and smooth, there exists a length scale $\ell_0$ so that, for all $\bs\in\Gamma$, the subset $\Gamma \cap B(\bs,\ell_0)$ has only one connected component. In what follows, we assume that the mesh is fine enough so that $h\le \ell_0$. This assumption implies that $T^\Gamma$ has a single connected component, and that the sub-cells $T^1$ and $T^2$ are connected. We also assume that $d(\Gamma,\partial\Omega)\ge 2h$.

\begin{figure}[htb]
\centering
\includegraphics[width=0.25\linewidth]{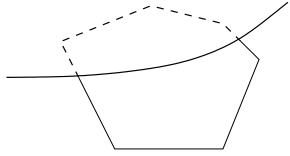}
\caption{Hexagonal cell $T$ cut by the interface $\Gamma$. The subdomain $\Omega^1$ is located below $\Gamma$, and the subdomain $\Omega^2$ is located above $\Gamma$. $(\partial T)^1$ is shown \cor{using solid lines} and $(\partial T)^2$ \cor{using dashed lines}; the sets $\calF_{(\dT)^1}$ and $\calF_{(\dT)^2}$ consist each of four elements, two of which are original faces of $T$ and two of which are sub-faces of the two faces of $T$ cut by the interface.}
\label{fig:cutcell1} 
\end{figure}

Let $l\in\polN$ be a polynomial degree and let $S$ be an $m$-dimensional affine manifold in $\Omega$ ($m\le d$); typically, $S$ is a mesh (sub-)cell (so that $m=d$) or a mesh (sub-)face (so that $m=d-1$). Then $\polP^l(S)$ denotes the space composed of the restriction to $S$ of $d$-variate polynomials of degree at most $l$.

\subsection{Mesh properties}

We make the following two assumptions on the mesh. Assumption~\ref{ass:Gamma} means that the interface is properly described by the mesh; this assumption is quantified by an interface regularity parameter $\gamma\in (0,1)$. Assumption~\ref{ass:ball} means that all the mesh cells are cut favorably by the interface; this property is quantified by a cut parameter $\delta\in (0,1)$. 
We will show in Section~\ref{sec:build_mesh} how to produce a shape-regular 
(polyhedral) mesh so that Assumption~\ref{ass:Gamma} and 
Assumption~\ref{ass:ball} hold true. The idea is that Assumption~\ref{ass:Gamma} can be satisfied by refining the mesh, whereas Assumption~\ref{ass:ball} can be satisfied by means of a local cell-agglomeration procedure.

\begin{assumption}[Resolving $\Gamma$] \label{ass:Gamma}
There is $\gamma\in (0,1)$ s.t.~for all $T\in\calT_h^\Gamma$, 
there is a point $\hat \bx_T \in \Real^d$ so that, for all $\bs\in T^\Gamma$,
$\|\hat\bx_T-\bs\|_{\ell^2}\le \gamma^{-1}h_T$ and $d(\hat \bx_T,T_\bs\Gamma)\ge \gamma h_T$ where $T_\bs\Gamma$ is the tangent plane to $\Gamma$ at the point $\bs$.
\end{assumption}

\begin{assumption}[Cut cells] \label{ass:ball}
There is $\delta\in (0,1)$ such that, for all $T\in\calT_h^\Gamma$ and all $i\in\{1,2\}$, there is $\tilde\bx_{T^i} \in T^i$ so that
\begin{equation} \label{eq:boule_in_Ti}
B(\tilde\bx_{T^i},\delta h_{T})\subset T^i. 
\end{equation} 
\end{assumption}

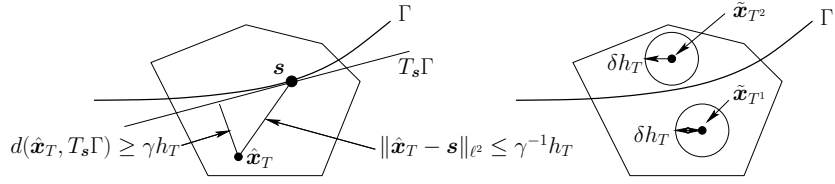
\begin{figure}[htb]
\centering
\scalebox{0.5}{\input{cutcell1_gd.pstex_t}}
\caption{\cor{Illustration of Assumption~\ref{ass:Gamma} (left) and of Assumption~\ref{ass:ball} (right) for an hexagonal cell $T$ cut by the interface $\Gamma$.}}
\label{fig:cutcell1_gd} 
\end{figure}

\subsection{Trace inequalities}

The purpose of Assumption~\ref{ass:Gamma} is to prove a multiplicative trace inequality that is needed to establish optimal approximation properties for the unfitted HHO method, whereas the purpose of Assumption~\ref{ass:ball} is to prove a discrete trace inequality that is needed in the stability analysis of the unfitted HHO method. Let us now prove these two trace inequalities.

\begin{lemma}[Multiplicative trace inequality] \label{lem:mult_tr}
There are real numbers $c_{\textup{mtr}}>0$ and $\theta_{\textup{mtr}}\ge1$, depending on the mesh regularity parameter $\rho\in (0,1)$ and the interface regularity parameter $\gamma\in (0,1)$, such that, for all $T\in\calT_h^\Gamma$, there is $\check\bx_T\in T$ so that, for all $i\in\{1,2\}$ and all $v\in H^1(\Tdg)$ with $\Tdg=B(\check\bx_T,\theta_{\textup{mtr}} h_T)$,
\begin{equation} \label{eq:mtrace}
\|v\|_{L^2(\partial T^i)} \le c_{\textup{mtr}}\, \bigg(h_T^{-\frac12} \|v\|_{L^2(\Tdg)} + \|v\|_{L^2(\Tdg)}^{\frac12}\|\GRAD v\|_{L^2(\Tdg)}^{\frac12} \bigg).
\end{equation}
\end{lemma}

\begin{proof}
The proof is inspired by the ideas from, e.g., \cite[Section 6]{WX10}. Let $T\in\calT_h^\Gamma$ and $i\in\{1,2\}$, and recall that $\partial T^i=(\partial T)^i \cup T^\Gamma$. We prove~\cor{\eqref{eq:mtrace}} for $v \in C^1(\Tdg)$ and then extend this bound to $H^1(\Tdg)$ by a density argument. Let us first bound $\|v\|_{L^2(T^\Gamma)}$. 
Integrating, for all $\bs\in T^\Gamma$, along the segment $\{\bp(\bs,t) := (1-t)\hat\bx_{T}+t \bs, \; \forall t\in [0,1]\}$, where the point $\hat\bx_T$ is given by Assumption~\ref{ass:Gamma}, we obtain
\[
v^2(\bs) = \int_{0}^1 \frac{\partial}{\partial t}\, \bigg( t^d\, v(\bp(\bs,t))^2\bigg)~\mbox{d}t, \qquad \forall \bs \in T^\Gamma.
\]
Integrating over $\bs\in T^\Gamma$ and developing the derivative with respect to $t$, we infer that
\[
\|v \|_{L^2(T^\Gamma)}^2= \int_{T^\Gamma} \int_{0}^1 \big(dt^{d-1}v(\bp(\bs,t))^2 + t^d v(\bp(\bs,t)) \nabla v(\bp(\bs,t)) {\cdot} (\bs - \hat\bx_{T})\big)~\mbox{d}t~\mbox{d}\bs.
\]
Let us introduce the cone $C(T)=\{\bp(\bs,t), \; \forall t\in[0,1],\; \forall \bs\in T^\Gamma\}$. Since $d(\hat\bx_{T},T_{\bs}\Gamma)\ge \gamma h_T$ (see Assumption~\ref{ass:Gamma}), the change of variable $d(\hat\bx_{T},T_{\bs}\Gamma)t^{d-1}\mbox{d}t\mbox{d}\bs = \mbox{d}\bp$ is legitimate. We then obtain that
\[
\|v\|_{L^2(T^\Gamma)}^2= \int_{C(T)} \big(dv(\bp(\bs,t))^2 + t v(\bp(\bs,t)) \nabla v(\bp(\bs,t)) {\cdot} (\bs - \hat \bx_{T})\big)~d(\hat\bx_{T},T_{\bs}\Gamma)^{-1}~\mbox{d}\bp.
\]
Since Assumption~\ref{ass:Gamma} implies that $C(T) \subset \Tdg_0:= B(\hat \bx_{T},\gamma^{-1}h_T)$, we conclude that
\[
\|v\|_{L^2(T^\Gamma)}^2 \le c_0 \bigg( h_T^{-1} \|v\|_{L^2(\Tdg_0)}^2 + \|v\|_{L^2(\Tdg_0)}\|\GRAD v\|_{L^2(\Tdg_0)} \bigg),
\]
where $c_0$ depends on the interface regularity parameter $\gamma\in (0,1)$. Let us now bound $\|v\|_{L^2(\dTi)}$. Proceeding as in \cite[Lemma~1.49]{DiPEr:12} using mesh regularity, we infer that there is a point $\check \bx_T\in T$ and positive real numbers $c_1,\theta_1$ depending on the mesh-regularity parameter $\rho\in (0,1)$ so that
 \[
\|v\|_{L^2(\dT)}^2 \le c_1 \bigg( h_T^{-1} \|v\|_{L^2(\Tdg_1)}^2 + \|v\|_{L^2(\Tdg_1)}\|\GRAD v\|_{L^2(\Tdg_1)} \bigg),
\]  
with $\Tdg_1=B(\check \bx_T,\theta_1h_T)$. To conclude, we combine the two above bounds using that $\Tdg_0\cup \Tdg_1 = B(\hat \bx_{T},\gamma^{-1}h_T) \cup B(\check \bx_T,\theta_1h_T) \subset B(\check \bx_T,\theta_{\textup{mtr}}h_T) =: \Tdg$ with $h_T^{-1}\|\hat\bx_T-\check\bx_T\|_{\ell^2} + \max(\gamma^{-1},\theta_1) \le 1+\gamma^{-1} + \max(\gamma^{-1},\theta_1) =:\theta_{\textup{mtr}}$ (since $\|\hat\bx_T-\check\bx_T\|_{\ell^2}\le \|\hat\bx_T-\bs\|_{\ell^2}+\|\bs-\check\bx_T\|_{\ell^2}\le 1+\gamma^{-1}$ for all $\bs\in T^\Gamma$), and we set $c_{\textup{mtr}}=\max(c_0,c_1)^{\frac12}$.
\end{proof}

\begin{lemma}[Discrete trace inequality]\label{lem:disc_trace}
Let $l\in\polN$, $l\ge0$. 
There is $c_{\textup{dtr}}$, depending on the polynomial degree $l$, the mesh regularity parameter $\rho\in (0,1)$, and the cut parameter $\delta\in (0,1)$, such that, for all $T\in\calT_h^\Gamma$, all $i\in\{1,2\}$, and all $v\in \polP^l(T^i)$, 
\begin{equation} \label{eq:trace}
\|v\|_{L^2(\dT^i)} \le c_{\textup{dtr}}\, h_T^{-\frac12} \|v\|_{L^2(T^i)}.
\end{equation}
\end{lemma}

\begin{proof}
Let $T\in\calT_h^\Gamma$. Let $i\in\{1,2\}$, and let $v\in \polP^l(T^i)$. Since $\partial T^i \subset B(\tilde\bx_{T^i},h_T)$ and $B(\tilde\bx_{T^i},\delta h_T) \subset T^i$ owing to~\eqref{eq:boule_in_Ti}, we observe that
\begin{align*}
\|v\|_{L^2(\partial T^i)} &\le |\partial T^i|^{\frac12} \|v\|_{L^\infty(\partial T^i)}\le |\partial T^i|^{\frac12} \|v\|_{L^\infty(B(\tilde\bx_{T^i},h_T))} \\
&\le \hat c \, |\partial T^i|^{\frac12} |B(\tilde\bx_{T^i},\delta h_T)|^{-\frac12} \|v\|_{L^2(B(\tilde\bx_{T^i},\delta h_T))} \\
&\le \hat c' \, |\partial T^i|^{\frac12} h_T^{-\frac{d}{2}} \|v\|_{L^2(B(\tilde\bx_{T^i},\delta h_T))}
\le \hat c' \, |\partial T^i|^{\frac12} h_T^{-\frac{d}{2}} \|v\|_{L^2(T^i)},
\end{align*}
where the factor $\hat c$ results from the inverse inequality $\|\hat v\|_{L^\infty(B(\bzero,1))} \le \hat c \, \|\hat v\|_{L^2(B(\bzero,\delta))}$ for all $\hat v \in \polP^l(B(\bzero,1))$ and the pullback using the bijective affine map from $B(\tilde\bx_{T^i},h_T)$ to $B(\bzero,1)$. We conclude by observing that $|\partial T^i|\le ch_T^{d-1}$ (with $c$ depending on $\rho$).
\end{proof}

\begin{remark}[Lemma~\ref{lem:disc_trace}]
For conforming finite elements on unfitted meshes, the discrete
trace inequality~\eqref{eq:trace} is invoked only on $T^\Gamma$.
Here, this inequality needs also to be invoked on $\dTi$ since the 
HHO method involves unknowns attached to the mesh faces, see
the proofs of Lemma~\ref{lem:stab_aT} and of Lemma~\ref{lem:consist} below.
\end{remark}

\section{The unfitted HHO method}\label{sec:HHO}

In this section, we describe the unfitted HHO method for the interface problem. Let $k\ge0$ be the polynomial degree.

\subsection{Uncut cells}

Let $\calT_h\uncut:=\calT_h^1 \cup \calT_h^2$ be the collection of the uncut cells. Let $T\in \calT_h\uncut$ and set $\kappa_T=\kappa^i$ if $T\in\calT_h^i$, $\iot$. We define the following local bilinear form for all $v,w\in H^1(T)$:
\begin{equation}
a_T(v,w) = \int_T \kappa_T \GRAD v\SCAL\GRAD w.
\end{equation}
The classical HHO method is defined locally on each uncut cell $T\in \calT_h\uncut$ from a pair of local unknowns which consist of one polynomial of order $(k+1)$ in $T$ and a piecewise polynomial of order $k$ on $\partial T$ (that is, one polynomial of order $k$ on each face $F\in\calF_{\dT}$). The local unknowns are generically denoted 
\begin{equation} \label{eq:def_XT_uncut}
\hat v_T=(v_T,v_{\partial T})\in \polP^{k+1}(T) \times \polP^k(\calF_{\dT}) =: \hat \calX_T\uncut,
\end{equation} 
with the piecewise polynomial space $\polP^k(\calF_{\dT}) = \vartimes_{F\in\calF_{\dT}} \polP^k(F)$. The placement of the discrete unknowns for the uncut cells is illustrated in Figure~\ref{fig:uncut}.

\begin{figure}[htb]
\begin{center}
\includegraphics[width=0.6\linewidth]{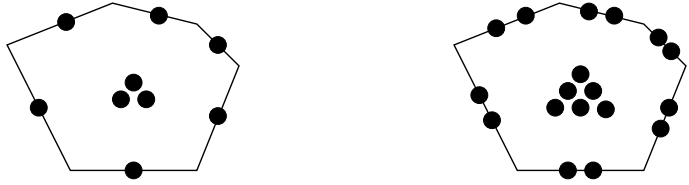}
\end{center}
\caption{Uncut hexagonal cell. Left: $k=0$; Right: $k=1$. Each dot attached to a geometric entity (face or cell) symbolizes one degree of freedom (not necessarily a pointwise evaluation).}
\label{fig:uncut}
\end{figure}

There are two key ingredients to devise the local HHO bilinear form. The first one is a reconstruction operator. Let $\hat v_T=(v_T,v_{\partial T})\in \hat \calX_T\uncut$. Then, we reconstruct a polynomial $r_T^{k+1}(\hat v_T) \in \polP^{k+1}(T)$ by requiring that, for all $z\in \polP^{k+1}(T)$, the following holds true:
\begin{equation} \label{eq:def_rT}
a_T(r_T^{k+1}(\hat v_T),z) = a_T(v_T,z) - \int_{\partial T} \kappa_T\GRAD z\SCAL\bn_T(v_T-v_{\partial T}),
\end{equation}
where $\bn_T$ is the unit outward-pointing normal to $T$. It is readily seen that $r_T^{k+1}(\hat v_T)$ is uniquely defined by~\eqref{eq:def_rT} up to an additive constant; one way to fix the constant is to prescribe $\int_T r_T^{k+1}(\hat v_T) = \int_T v_T$ (this choice is irrelevant in what follows). The second ingredient is the stabilization bilinear form defined so that, for all $\hat v_T,\hat w_T \in \hat \calX_T\uncut$,
\begin{equation} \label{eq:def_sT}
s_T(\hat v_T,\hat w_T) = \kappa_Th_T^{-1} \int_{\partial T} \Pi^k_{\partial T}(v_T-v_{\partial T})(w_T-w_{\partial T}),
\end{equation} 
where $\Pi^k_{\partial T}$ denotes the $L^2$-orthogonal projector onto the piecewise polynomial space $\polP^k(\calF_{\dT})$. Finally, the local HHO bilinear and linear forms to be used when assembling the global discrete problem (see Section~\ref{sec:global_disc}) are as follows: For all $\hat v_T,\hat w_T\in \hat \calX_T\uncut$,
\begin{subequations} \label{eq:def_al_uncut}
\begin{align}
\hat a_T\uncut(\hat v_T,\hat w_T) = {}& a_T(r_T^{k+1}(\hat v_T),r_T^{k+1}(\hat w_T)) + s_T(\hat v_T,\hat w_T), \\
\hat \ell_T\uncut(\hat w_T) = {}& \int_{T} fw_T.
\end{align}
\end{subequations}

\begin{remark}[Cell unknowns]
In the classical HHO method, there is some flexibility in the choice of the cell unknowns since one can take them to be polynomials of order $l\in\{k-1,k,k+1\}$. In the present context, we will need to work with polynomials of order $(k+1)$ in the cut cells \cor{to achieve optimal approximation properties (see Section~\ref{sec:approx})}; for simplicity, we consider polynomials of order $(k+1)$ in the uncut cells as well. Taking polynomials of order $l\in\{k-1,k\}$ in the uncut cells leads to slightly smaller matrices to be inverted when computing the reconstruction operator from~\eqref{eq:def_rT}, but requires a somewhat more involved design of the stabilization operator than in~\eqref{eq:def_sT} (see \cite{DiPEL:14,DiPEr:15}).
\end{remark}

\subsection{Cut cells}

Let $T\in\calT_h^\Gamma$. We use capital letters to denote a generic pair $V=(v^1,v^2) \in H^1(T^1)\times H^1(T^2)$. We define the following Nitsche-mortaring bilinear form for all $V,W\in H^{s}(T^1)\times H^1(T^2)$, $s>\frac32$:
\begin{subequations} \label{eq:nT} \begin{align}
n_T(V,W) &= \sum_{\iot} \int_{T^i} \kappa^i\GRAD v^i\SCAL\GRAD w^i + n_{T\Gamma}(V,W), \label{eq:def_nT} \\
n_{T\Gamma}(V,W) &= - \int_{T^\Gamma} (\kappa\GRAD v)^1\SCAL\bn_\Gamma\jump{W}_\Gamma + (\kappa\GRAD w)^1\SCAL\bn_\Gamma\jump{V}_\Gamma - \eta \frac{\kappa^1}{h_T} \jump{V}_\Gamma\jump{W}_\Gamma, \label{eq:def_nTG}
\end{align} \end{subequations}
where the user-specified parameter $\eta$ is such that $\eta \ge 4 c_{\textup{dtr}}^2$ where $c_{\textup{dtr}}$ results from the discrete trace inequality~\eqref{eq:trace} with polynomial degree $l=k$. Note also that the jump-penalty term is weighted by the lowest value of the diffusion coefficient. 

We consider a quadruple of discrete HHO unknowns, $\hat V_T = (V_T,V_{\dT})$, where both $V_T$ and $V_{\dT}$ are pairs associated with the partition $\overline\Omega = \overline{\Omega^1} \cup \overline{\Omega^2}$, so that 
\begin{equation}
V_T=(v_{T^1},v_{T^2}) \in \polP^{k+1}(T^1) \times \polP^{k+1}(T^2),
\end{equation}
and
\begin{equation}
V_{\dT}=(v_{(\dT)^1},v_{(\dT)^2}) \in \polP^{k}(\calF_{(\dT)^1}) \times \polP^{k}(\calF_{(\dT)^2}),
\end{equation}
where $\polP^k(\calF_{(\dT)^i}) := \vartimes_{F\in\calF_{(\dT)^i}} \polP^k(F)$ is the piecewise polynomial space of order $k$ on $(\partial T)^i$ based on the (sub-)faces in $\calF_{(\dT)^i}$. (Recall that, by definition, all the elements $F$ of $\calF_{(\dT)^i}$ are subsets of $\dTi=\partial T\cap \Omega^i$.) Note that we do not introduce any discrete unknown on $T^\Gamma$. We use the concise notation $\hat V_T\in \hat \calX_T\cut$ with
\begin{align} \label{eq:def_XT_int}
\hat \calX_T\cut &= \Big( \polP^{k+1}(T^1) \times \polP^{k+1}(T^2) \Big) \; \times \; \Big( \polP^{k}(\calF_{(\dT)^1}) \times \polP^{k}(\calF_{(\dT)^2}) \Big).
\end{align}
The placement of the discrete HHO unknowns in the cut cells for the interface problem is illustrated in Figure~\ref{fig:cutcell3}. 

\begin{figure}[htb]
\begin{center}
\includegraphics[width=0.6\linewidth]{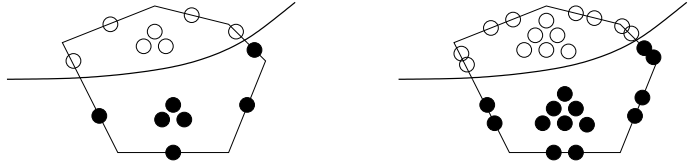}
\end{center}
\caption{Cut hexagonal cell for the interface problem. The subdomain $\Omega^1$ is located below $\Gamma$ \cor{with the corresponding HHO unknowns shown by filled circles}, and the subdomain $\Omega^2$ is located above $\Gamma$ \cor{with the corresponding HHO unknowns shown by empty circles}. Left: $k=0$; Right: $k=1$.}
\label{fig:cutcell3}
\end{figure}

As above, there are two key ingredients to devise the local HHO bilinear form: reconstruction and stabilization. Let $\hat V_T\in \hat\calX_T\cut$. We reconstruct a pair of polynomials $R_T^{k+1}(\hat V_T) \in \polP^{k+1}(T^1) \times \polP^{k+1}(T^2)$ by requiring that, for all $Z=(z^1,z^2)\in \polP^{k+1}(T^1) \times \polP^{k+1}(T^2)$, the following holds true:
\begin{align}
n_T(R_T^{k+1}(\hat V_T),Z) = {}& n_T(V_T,Z) - \sum_{\iot} \int_{(\partial T)^i} \kappa^i \GRAD z^i\SCAL \bn_T(v_{T^i}-v_{(\dT)^i}). \label{eq:def_RT}
\end{align}
It follows from Lemma~\ref{lem:stab_bnd_nT} below that $R_T^{k+1}(\hat V_T)$ is uniquely defined by~\eqref{eq:def_RT} up to the same additive constant for both of its components; one way to fix the constant is to prescribe $\sum_{\iot} \int_{T^i} (R_T^{k+1}(\hat V_T))^i = \sum_{\iot} \int_{T^i} v_{T^i}$ (this choice is irrelevant in what follows). Concerning stabilization, we set for all $\hat V_T,\hat W_T\in\hat \calX_T\cut$,
\begin{equation} \label{eq:def_sT_cut}
s_T(\hat V_T,\hat W_T) = \sum_{\iot} \kappa^i h_T^{-1} \int_{\dTi} \Pi_{\dTi}^k(v_{T^i}-v_{(\dT)^i})(w_{T^i}-w_{(\dT)^i}),
\end{equation}
where $\Pi_{(\partial T)^i}^k$ denotes the $L^2$-orthogonal projector onto the piecewise polynomial space $\polP^k(\calF_{(\dT)^i})$. Finally, the local HHO bilinear and linear forms are as follows: For all $\hat V_T,\hat W_T\in \hat\calX_T\cut$,
\begin{subequations} \label{eq:def_al_int}
\begin{align}
\hat a_T\cut(\hat V_T,\hat W_T) = {}& n_T(R_T^{k+1}(\hat V_T),R_T^{k+1}(\hat W_T)) + s_T(\hat V_T,\hat W_T), \\
\hat \ell_T\cut(\hat W_T) = {}& \sum_{\iot} \int_{T^i} fw_{T^i} + \int_{T^\Gamma} (g_{\textup{N}}w_{T^2} + g_{\textup{D}} \phi_T(W_T)), 
\end{align}
\end{subequations}
with $\phi_T(W_T) = -\kappa^1\GRAD w_{T^1}\SCAL\bn_\Gamma + \eta\kappa^1h_T^{-1}\jump{W_T}_\Gamma$ (the definition of the integral over $T^\Gamma$ follows from consistency arguments, see the proof of Lemma~\ref{lem:consist} below).

\subsection{The global discrete problem}
\label{sec:global_disc}

The mesh faces are collected in the set $\calFh$ which is partitioned into $\calFh = \calFh^1\cup \calFh^\Gamma\cup \calFh^2$, where $\calFh^i$, $\iot$, collect the mesh faces inside the subdomain $\Omega^i$ and $\calFh^\Gamma$ collects the mesh faces cut by the interface. 
We also define for all $\iot$,
\begin{subequations}
\begin{align}
\hat \calT_h^i &:= \calT_h^i \cup \{T^i=T\cap\Omega^i\tq T\in\calT_h^\Gamma\}, \\
\hcalFh^i &:= \calFh^i \cup \{F^i=F\cap\Omega^i\tq F\in\calFh^\Gamma\},
\end{align}
\end{subequations}
\ie $\hat \calT_h^i$ (resp., $\hcalFh^i$) is the collection of all the mesh cells (resp., faces) inside $\Omega^i$ plus the collection of the sub-cells (resp., sub-faces) of the cut cells (resp., cut faces) inside $\Omega^i$. Let us set
\begin{equation}
\hat \calX_h^i := \vartimes_{T\in \hat \calT_h^i} \polP^{k+1}(T) \; \times \; \vartimes_{F\in \hcalFh^i} \polP^{k}(F).
\end{equation}

The global discrete space is $\hat \calX_h:=\hat \calX_h^1\times \hat \calX_{h}^2$. Let $\calFh^\partial$ be the collection of the mesh faces located at the boundary $\partial\Omega$ (note that the faces in $\calFh^\partial$ are in one and only one of the subsets $\calFh^i$, but not in $\calFh^\Gamma$ since the interface $\Gamma$ is located in the interior of $\Omega$). We enforce the homogeneous Dirichlet condition on $\partial\Omega$ by zeroing out the discrete HHO unknowns attached to the mesh faces in $\calFh^\partial$. 
Let $i^\partial\in\{1,2\}$ be the index of the subdomain touching the boundary $\partial\Omega$. Let $\hat \calX_{h0}^{i^\partial}$ be the subspace of $\hat \calX_h^{i^\partial}$ composed of all the discrete HHO unknowns such that their component attached to a mesh face is zero if this face lies on the boundary $\partial \Omega$. 
If $i^\partial=1$, we set $\hat \calX_{h0}:=\hat \calX_{h0}^1\times \hat \calX_{h}^2$; otherwise, we set $\hat \calX_{h0}:=\hat \calX_h^1\times \hat \calX_{h0}^2$. 

Let $\hat V_h\in \hat \calX_{h0}$. For all $T\in\calT_h\uncut = \calT_h^1\cup \calT_h^2$, we denote $\hat v_T=(v_T,v_{(\dT)})\in\hat\calX_T\uncut$ (see~\eqref{eq:def_XT_uncut}) the components of $\hat V_h$ attached to the cell $T$. For all $T\in\calT_h^\Gamma$, we denote $\hat V_T=(V_T,V_{\dT})\in\hat\calX_T\cut$ (see~\eqref{eq:def_XT_int}) the components of $\hat V_h$ attached to the cell $T$. The discrete problem we want to solve reads as follows: Find $\hat U_h\in \hat \calX_{h0}$ s.t.
\begin{equation} \label{eq:discrete_pb}
\hat a_h(\hat U_h,\hat W_h) = \hat \ell_h(\hat W_h), \qquad \forall \hat W_h\in \hat \calX_{h0},
\end{equation}
with
\begin{subequations}
\begin{align}
\hat a_h(\hat V_h,\hat W_h) &= \sum_{T\in\calT_h\uncut} \hat a_T\uncut(\hat v_T,\hat w_T) + \sum_{T\in\calT_h^\Gamma} \hat a_T\cut(\hat V_T,\hat W_T), \\ 
\hat \ell_h(\hat W_h) &= \sum_{T\in\calT_h\uncut} \hat \ell_T\uncut(w_T) + \sum_{T\in\calT_h^\Gamma} \hat \ell_T\cut(\hat W_T),
\end{align}
\end{subequations}
where $\hat a_T\uncut(\SCAL,\SCAL)$ and $\hat \ell_T\uncut(\SCAL)$ are defined by~\eqref{eq:def_al_uncut} for all $T\in\calT_h\uncut$ and $\hat a_T\cut(\SCAL,\SCAL)$ and $\hat \ell_T\cut(\SCAL)$ are defined by~\eqref{eq:def_al_int} for all $T\in\calT_h^\Gamma$.

The discrete problem~\eqref{eq:discrete_pb} can be solved efficiently by eliminating locally all the cell unknowns using static condensation. This local elimination leads to a global transmission problem on the mesh skeleton involving only the face unknowns with a stencil that couples unknowns attached to neighboring faces (in the sense of cells). Once this global transmission problem is solved, the cell unknowns are recovered by local solves. We refer the reader, e.g., to \cite{CiDPE:17} for more details in the case of classical HHO methods.

\section{Analysis}
\label{sec:proof}

In this section we analyze the convergence of the unfitted HHO method for the interface problem. The proof consists in establishing stability, consistency, and boundedness properties for the discrete forms $\hat a_h$ and $\hat\ell_h$, and in devising a local approximation operator related to the local reconstruction operators $r_T^{k+1}$ (see~\eqref{eq:def_rT}) and $R_T^{k+1}$ (see~\eqref{eq:def_RT}). The mesh $\calT_h$ is assumed to satisfy Assumption~\ref{ass:Gamma} and Assumption~\ref{ass:ball} so as to invoke the trace inequalities from Lemma~\ref{lem:mult_tr} and~\ref{lem:disc_trace}.  

In what follows, we often abbreviate $A\lesssim B$ the inequality $A\le CB$ for positive real numbers $A$ and $B$, where the constant $C$ does not depend on $\kappa$ nor on the way the interface cuts the mesh-cells, but only depends on the polynomial degree $k\ge0$, the mesh regularity parameter $\rho\in(0,1)$, the interface regularity parameter $\gamma\in(0,1)$ from Assumption~\ref{ass:Gamma}, and the cut parameter $\delta\in (0,1)$ from Assumption~\ref{ass:ball}. 

\subsection{Stability and well-posedness}

We start with the following stability and boundedness results on the Nitsche-mortaring bilinear form $n_T$ defined by~\eqref{eq:nT} for all $T\in\calT_h^\Gamma$. We define the following stability semi-norm for all $V=(v^1,v^2) \in H^1(T^1)\times H^1(T^2)$:
\begin{equation} \label{eq:def_seminorm_nT}
|V|_{n_T}^2 := \sum_{\iot} \kappa^i \|\GRAD v^i\|_{T^i}^2 + \eta \frac{\kappa^1}{h_T} \|\jump{V}_\Gamma\|_{T^\Gamma}^2.
\end{equation} 
Recall our assumption on the penalty parameter $\eta\ge 4 c_{\textup{dtr}}^2$.

\begin{lemma}[Stability and boundedness of $n_T$] \label{lem:stab_bnd_nT}
Let $T\in\calT_h^\Gamma$. The following holds true for all $V \in \polP^{k+1}(T^1)\times \polP^{k+1}(T^2)$:
\begin{equation} \label{eq:nT_coer}
n_T(V,V) \ge \frac12 |V|_{n_T}^2.
\end{equation}
Moreover, the following holds true for all $V,W \in \polP^{k+1}(T^1)\times \polP^{k+1}(T^2)$:
\begin{equation} \label{eq:nT_bnd_disc}
|n_T(V,W)| \lesssim |V|_{n_T} |W|_{n_T},
\end{equation}
and for all $V \in H^s(T^1)\times H^1(T^2)$, $s>\frac32$, and all $W\in \polP^{k+1}(T^1) \times \polP^{k+1}(T^2)$:
\begin{equation} \label{eq:nT_bnd}
|n_T(V,W)| \lesssim |V|_{n_T\sharp}|W|_{n_T}, \qquad
|V|_{n_T\sharp}^2 := |V|_{n_T}^2  + \kappa^1 h_T \|\GRAD v^1\|_{T^\Gamma}^2.
\end{equation}
\end{lemma}

\begin{proof}
The proof is classical; we sketch it for completeness. Let $V \in \polP^{k+1}(T^1)\times \polP^{k+1}(T^2)$, and let us set $\xi=(\sum_{\iot} \kappa^i \|\GRAD v^i\|_{T^i}^2)^{\frac12}$ and $\zeta = (\eta \frac{\kappa^1}{h_T} \|\jump{V}_\Gamma\|_{T^\Gamma}^2)^{\frac12}$ so that $|V|_{n_T}^2 = \xi^2 + \zeta^2$. The definition~\eqref{eq:nT} of $n_T$ followed by the Cauchy--Schwarz inequality and the discrete trace inequality~\eqref{eq:trace} (applied on $T^\Gamma$ with $l=k$) yields
\[
n_T(V,V) = \xi^2 -2\int_{T^\Gamma} (\kappa\GRAD v)^1\SCAL\bn_\Gamma\jump{V}_\Gamma + \zeta^2 \ge \xi^2 -2c_{\textup{dtr}}\eta^{-\frac12}\xi\zeta + \zeta^2,
\]
so that $n_T(V,V)\ge\frac12(\xi^2+\zeta^2)$ (i.e., \eqref{eq:nT_coer}) follows from the assumption that $\eta\ge 4 c_{\textup{dtr}}^2$. Moreover, using the Cauchy--Schwarz inequality, we infer that
\[
|n_T(V,W)| \le |V|_{n_T}|W|_{n_T} + \kappa^1\|\GRAD v^1\|_{T^\Gamma}\|\jump{W}_\Gamma\|_{T^\Gamma} + \kappa^1\|\GRAD w^1\|_{T^\Gamma}\|\jump{V}_\Gamma\|_{T^\Gamma},
\]
so that \eqref{eq:nT_bnd_disc} and~\eqref{eq:nT_bnd} follow from the discrete trace inequality~\eqref{eq:trace}.
\end{proof}

We can now address the stability of the local HHO bilinear forms $\hat a_T\uncut$ and $\hat a_T\cut$.
For all $T\in\calT_h\uncut$, we consider the local semi-norm used in the analysis of classical HHO methods: For all $\hat v_T=(v_T,v_{\dT})\in\hat\calX_T\uncut$,
\begin{align}
|\hat v_T|_{\hat a_T}^2 :={}& \kappa_T \|\GRAD v_T\|_{T}^2 + \kappa_Th_T^{-1}\|\Pi_{\dT}^k(v_{T}-v_{\dT})\|_{\dT}^2 = |v_T|_{a_T}^2 + s_T(\hat v_T,\hat v_T), \label{eq:def_seminorm_aT_uncut}
\end{align}
where we have set $|v_T|_{a_T}^2 := \kappa_T \|\GRAD v_T\|_{T}^2$.
For all $T\in\calT_h^\Gamma$, we define the following local semi-norm: For all $\hat V_T=(V_T,V_{\dT}) = ((v_{T^1},v_{T^2}),(v_{(\dT)^1},v_{(\dT)^2}))\in \hat\calX_T\cut$:
\begin{align} 
|\hat V_T|_{\hat a_T}^2 := {}&\sum_{\iot} \kappa^i \|\GRAD v_{T^i}\|_{T^i}^2 + \eta \frac{\kappa^1}{h_T} \|\jump{V_T}_\Gamma\|_{T^\Gamma}^2 \nonumber \\&+ \sum_{\iot} \kappa^ih_T^{-1}\|\Pi_{\dTi}^k(v_{T^i}-v_{(\dT)^i})\|_{\dTi}^2 
= |V_T|_{n_T}^2 + s_T(\hat V_T,\hat V_T). \label{eq:def_seminorm_aT}
\end{align}

\begin{lemma}[Stability] \label{lem:stab_aT}
The following holds true: 
\begin{subequations} \begin{alignat}{3} \label{eq:stab_uncut}
\hat a_T\uncut(\hat v_T,\hat v_T) &\gtrsim |\hat v_T|_{\hat a_T}^2,
&\qquad &\forall T\in\calT_h\uncut, &\; &\forall \hat v_T\in \hat \calX_T\uncut, \\
\label{eq:stab_cut}
\hat a_T\cut(\hat V_T,\hat V_T) &\gtrsim |\hat V_T|_{\hat a_T}^2,
&\qquad &\forall T\in\calT_h^\Gamma, &\; &\forall \hat V_T \in \hat \calX_T\cut.
\end{alignat} \end{subequations}
\end{lemma}

\begin{proof}
The proof of~\eqref{eq:stab_uncut} follows from \cite[Lemma 4]{DiPEL:14}.
Let us now prove~\eqref{eq:stab_cut}. Let $T\in\calT_h^\Gamma$ and let
$\hat V_T \in \hat \calX_T\cut$. 
Taking $Z=V_T=(v_{T^1},v_{T^2})$ in the definition~\eqref{eq:def_RT} of the reconstruction operator and using the stability of $n_T$ from Lemma~\ref{lem:stab_bnd_nT}, we infer that
\begin{align*}
|V_T|_{n_T}^2 \lesssim {}& n_T(V_T,V_T) \\ = {}&n_T(R_T^{k+1}(\hat V_T),V_T)
+ \sum_{\iot} \int_{(\partial T)^i}  \kappa^i \GRAD v_{T^i}\SCAL \bn_T(v_{T^i}-v_{(\dT)^i}).
\end{align*}
The first term on the right-hand side is controlled using the boundedness property~\eqref{eq:nT_bnd_disc} of $n_T$ and Young's inequality to hide $|V_T|_{n_T}$ on the left-hand side. For the second term, we use the Cauchy--Schwarz inequality, the fact that $(\kappa^i \GRAD v_{T^i}\SCAL \bn_T)_{|\dTi} \in \polP^k(\calF_{(\dT)^i})$, and the definition~\eqref{eq:def_sT_cut} of $s_T(\SCAL,\SCAL)$ to obtain
\[
\int_{(\partial T)^i}  \kappa^i \GRAD v_{T^i}\SCAL \bn_T(v_{T^i}-v_{(\dT)^i})
\le (\kappa^i)^{\frac12}h_T^{\frac12}\|\GRAD v_{T^i}\|_{\dTi} s_T(\hat V_T,\hat V_T)^{\frac12}.
\]
Then, we invoke the discrete trace inequality~\eqref{eq:trace} on $(\partial T)^i\subset \partial T^i$ for all $\iot$, and Young's inequality to hide $(\kappa^i)^{\frac12}\|\GRAD v_{T^i}\|_{T^i}$ on the left-hand side. Putting everything together, we infer that
\[
|V_T|_{n_T}^2 \lesssim |R_T^{k+1}(\hat V_T)|_{n_T}^2 + s_T(\hat V_T,\hat V_T),
\]
so that using \eqref{eq:def_seminorm_aT} and the stability of $n_T$ from Lemma~\ref{lem:stab_bnd_nT}, we conclude that
\begin{align*}
|\hat V_T|_{\hat a_T}^2 &= |V_T|_{n_T}^2 + s_T(\hat V_T,\hat V_T) \\
&\lesssim |R_T^{k+1}(\hat V_T)|_{n_T}^2 + s_T(\hat V_T,\hat V_T)\\
&\lesssim n_T(R_T^{k+1}(\hat V_T),R_T^{k+1}(\hat V_T)) + s_T(\hat V_T,\hat V_T)
= \hat a_T(\hat V_T,\hat V_T),
\end{align*}
which is the expected estimate.
\end{proof}

Summing the local semi-norms over the mesh cells, we define, for all $\hat V_h\in \hat \calX_h$,
\begin{equation}
|\hat V_h|_{\hat a_h}^2 := \sum_{T\in\calT_h\uncut} |\hat v_T|_{\hat a_T}^2 + \sum_{T\in\calT_h^\Gamma} |\hat V_T|_{\hat a_T}^2. \label{eq:def_norm_ah}
\end{equation}
Note that $|\cdot|_{\hat a_h}$ defines a norm on the subspace $\hat \calX_{h0}$. Indeed, assume that $|\hat V_h|_{\hat a_h}=0$ for some $\hat V_h\in \hat \calX_{h0}$. Then, for all $T\in \calT_h^\Gamma$, we have $|V_T|_{n_T}=0$ and $s_T(\hat V_T,\hat V_T)=0$. The nullity of the first term implies that $v_{T^1}$ and $v_{T^2}$ are constant functions that take the same value, and the nullity of the second term implies that $v_{(\dT)^1}$ and $v_{(\dT)^2}$ are also constant functions that take the same value as $v_{T^1}$ and $v_{T^2}$. Moreover, for all $T\in\calT_h\uncut$, $|\hat v_T|_{\hat a_T}=0$ implies that $v_T$ and $v_{\dT}$ take the same constant value. We can then propagate this constant value up to the boundary $\partial \Omega$ where the components of $\hat V_h$ attached to the boundary faces vanish. Thus, all the components of $\hat V_h$ are zero.

\begin{corollary}[Well-posedness]
The discrete problem~\eqref{eq:discrete_pb} is well-posed.
\end{corollary}

\begin{proof}
We apply the Lax--Milgram Lemma.
\end{proof}

\subsection{Approximation} \label{sec:approx}

Let $u$ be the exact solution with $u^i:=u_{|\Omega^i}$, for all $\iot$. 
We set $U\upex=(u^1,u^2)\in H^1(\Omega^1)\times H^1(\Omega^2)$.

\subsubsection{Uncut cells}

Let $T\in\calT_h\uncut$. We set $u\upex_T=u^i_{|T}$ where $\iot$ is s.t.~$T\in\calT_h^i$ and we consider the approximation of $u\upex_T$ in $T$ defined by
\begin{equation}
j^{k+1}_T(u\upex_T)=\Pi^{k+1}_T(u\upex_T),
\end{equation}
where $\Pi^{k+1}_T$ stands for the $L^2$-orthogonal projector onto $\polP^{k+1}(T)$ (we use a specific notation $j^{k+1}_T$ for similarity with cut cells, see below). We introduce the following local norm: For all $v\in H^s(T)$, $s>\frac32$:
\begin{equation}
\|v\|_{*T}^2 = \kappa_T\big(\|\GRAD v\|_{T}^2+h_T\|\GRAD v\|_{\dT}^2+h_T^{-1}\|v\|_{\dT}^2\big).
\end{equation}

\begin{lemma}[Approximation by $j^{k+1}_T$] \label{lem:app_j_uncut}
Assume $U\upex \in H^{k+2}(\Omega^1)\times H^{k+2}(\Omega^2)$.
The following holds true for all $T\in\calT_h\uncut$:
\begin{equation}
\|j^{k+1}_T(u\upex_T)-u\upex_T\|_{*T} \lesssim \kappa_T^{\frac12} h_T^{k+1} |u\upex_T|_{H^{k+2}(T)}.
\end{equation}
\end{lemma}

\begin{proof}
The approximation properties of the $L^2$-orthogonal projector are classical on meshes where all the cells can be mapped to a reference cell, see, e.g., \cite{EG04}. On meshes with polyhedral cells which can be split into a finite number of shape-regular simplices, one can proceed as in the proof of \cite[Lem.~5.4]{ErnGuermond:17} by combining the Poincar\'e--Steklov inequality in each sub-simplex and the multiplicative trace inequality. 
\end{proof}

Let us now define
\begin{equation} \label{eq:def_p_T}
p^{k+1}_T(u\upex_T)=r^{k+1}_T(\hat \jmath^{k+1}_T(u\upex_T))\in \polP^{k+1}(T),
\end{equation}
where $r^{k+1}_T$ is the reconstruction operator defined by~\eqref{eq:def_rT} and
\begin{equation}
\hat \jmath^{k+1}_T(u\upex_T) = (j^{k+1}_T(u\upex_T),\Pi_{\dT}^k(u\upex_T))
= (\Pi^{k+1}_T(u\upex_T),\Pi_{\dT}^k(u\upex_T))\in\hat\calX_T\uncut,
\end{equation}
where $\Pi_{\dT}^k$ stands for the $L^2$-orthogonal projector onto the piecewise polynomial space $\polP^k(\calF_{\dT})$. 

\begin{lemma}[Approximation] \label{lem:approx_uncut}
Assume that $U\upex\in H^s(\Omega^1)\times H^s(\Omega^2)$, $s>\frac32$.
The following holds true for all $T\in\calT_h\uncut$:
\begin{equation}
|p^{k+1}_T(u\upex_T)-u\upex_T|_{a_T} + s_T(\hat \jmath^{k+1}_T(u\upex_T),\hat \jmath^{k+1}_T(u\upex_T))^{\frac12} \lesssim \|j^{k+1}_T(u\upex_T)-u\upex_T\|_{*T}.
\end{equation}
\end{lemma}

\begin{proof}
It is shown in \cite[Lemma 3]{DiPEL:14} that $p^{k+1}_T(u\upex_T)$ is the elliptic projector of $u\upex_T$ onto $\polP^{k+1}(T)$, so that $a_T(p^{k+1}_T(u\upex_T)-u\upex_T,w)=0$ for all $w\in \polP^{k+1}(T)$, and $|p^{k+1}_T(u\upex_T)-u\upex_T|_{a_T}\le |j^{k+1}_T(u\upex_T)-u\upex_T|_{a_T}$. The bound on $|p^{k+1}_T(u\upex_T)-u\upex_T|_{a_T}$ then follows from 
$|\cdot|_{a_T}\le \|\cdot\|_{*T}$.
To bound $s_T(\hat \jmath^{k+1}_T(u\upex_T),\hat \jmath^{k+1}_T(u\upex_T))$, we proceed as in the proof of~\eqref{eq:bnd_sT_hatJ} below.
\end{proof}

\subsubsection{Cut cells}
For all $T\in\calT_h^\Gamma$, let us define the pair
\begin{equation} \label{eq:def_Uupex_T}
U\upex_T=(u^1_{|T^1},u^2_{|T^2}) \in H^1(T^1)\times H^1(T^2).
\end{equation} 
Let $E^i : H^1(\Omega^i) \to H^1(\Real^d)$, for all $\iot$, be stable extension operators. Recall the ball $\Tdg$ introduced in Lemma~\ref{lem:mult_tr} and observe that $T\subset \Tdg$ since $\theta_{\textup{mtr}}\ge1$. We construct an approximation of the pair $U\upex_T$ in $T$ by setting
\begin{equation} \label{eq:def_JT}
J_T^{k+1}(U\upex) := (\Pi_{\Tdg}^{k+1}(E^1(u^1))_{|T^1},\Pi_{\Tdg}^{k+1}(E^2(u^2))_{|T^2})
\in \polP^{k+1}(T^1) \times \polP^{k+1}(T^2),
\end{equation}
where $\Pi_{\Tdg}^{k+1}$ stands for the $L^2$-orthogonal projector onto $\polP^{k+1}(\Tdg)$ (we do not project using the set $T^i$, but the larger set $\Tdg$, to avoid dealing with approximation properties on $T^i$). We introduce the following local norm: For all $V=(v^1,v^2)\in H^s(T^1)\times H^s(T^2)$, $s>\frac32$,
\begin{align}
\|V\|_{*T}^2 = {}& \sum_{\iot} \kappa^i\big( \|\GRAD v^i\|_{T^i}^2 + h_T\|\GRAD v^i\|_{\dTi}^2 + h_T^{-1} \|v^i\|_{\dTi}^2 \big) \nonumber \\
& + \kappa^1 \big(h_T\|\GRAD v^1\|_{T^\Gamma}^2 + h_T^{-1}\|\jump{V}_\Gamma\|_{T^\Gamma}^2  \big) + \kappa^2h_T \|\GRAD v^2\|_{T^\Gamma}^2. \label{eq:def_cut_*T}
\end{align}
Note that $|V|_{n_T} \le |V|_{n_T\sharp} \le \|V\|_{*T}$.

\begin{lemma}[Approximation by $J^{k+1}_T$] \label{lem:app}
Assume $U\upex \in H^{k+2}(\Omega^1)\times H^{k+2}(\Omega^2)$.
The following holds true for all $T\in\calT_h^\Gamma$:
\begin{equation}
\|J_T^{k+1}(U\upex)-U\upex_T\|_{*T} \lesssim \sum_{\iot} (\kappa^i)^{\frac12} h_T^{k+1}|E^i(u^i)|_{H^{k+2}(\Tdg)}.
\end{equation}
\end{lemma}

\begin{proof}
We need to bound the six terms on the right-hand side of~\eqref{eq:def_cut_*T}. The bound on the norm on $T^i$ is straightforward since this norm can be bounded by the norm on $\Tdg$ where we can use the classical approximation properties of $\Pi_{\Tdg}^{k+1}$ (recall that $T\subset \Tdg$). To bound the three norms on $(\dT)^i$ and the two norms on $T^\Gamma$, we use the multiplicative trace inequality from Lemma~\ref{lem:mult_tr} and the approximation properties of $\Pi_{\Tdg}^{k+1}$ on $\Tdg$.
\end{proof}

Let us now define
\begin{equation}
P_T^{k+1}(U\upex)=R_T^{k+1}(\hat J_T^{k+1}(U\upex)) \in \polP^{k+1}(T^1) \times \polP^{k+1}(T^2), \label{eq:def_AT}
\end{equation}
where $R^{k+1}_T$ is the reconstruction operator defined by~\eqref{eq:def_RT} and
\begin{equation}
\hat J_T^{k+1}(U\upex)
:=(J_T^{k+1}(U\upex),(\Pi_{(\partial T)^1}^k(u^1),\Pi_{(\partial T)^2}^k(u^2))) \in \hat\calX_T\cut, \label{eq:def_hat_JT}
\end{equation}
so that $\hat J_T^{k+1}(U\upex) = ((\Pi_{\Tdg}^{k+1}(E^1(u^1))_{|T^1},\Pi_{\Tdg}^{k+1}(E^2(u^2))_{|T^2}),(\Pi_{(\partial T)^1}^k(u^1),\Pi_{(\partial T)^2}^k(u^2)))$. 

\begin{lemma}[Approximation] \label{lem:prelim}
Assume that $U\upex\in H^s(\Omega^1)\times H^s(\Omega^2)$, $s>\frac32$.
The following holds true for all $T\in\calT_h^\Gamma$: \textup{(i)} 
For all $W_T \in \polP^{k+1}(T^1) \times \polP^{k+1}(T^2)$,
\begin{equation} \label{eq:bnd_sup_nT}
n_T(P_T^{k+1}(U\upex)-U\upex_T,W_T) \lesssim \|J_T^{k+1}(U\upex)-U\upex_T\|_{*T}  |W_T|_{n_T}.
\end{equation}
\textup{(ii)} We have
\begin{equation} \label{eq:bnd_Au-u_nT}
|P_T^{k+1}(U\upex)-U\upex_T|_{n_T} \lesssim \|J_T^{k+1}(U\upex)-U\upex_T\|_{*T}.
\end{equation}
\textup{(iii)} We have
\begin{equation} \label{eq:bnd_sT_hatJ}
s_T(\hat J_T^{k+1}(U\upex),\hat J_T^{k+1}(U\upex))^{\frac12} \lesssim \|J_T^{k+1}(U\upex)-U\upex_T\|_{*T}.
\end{equation}
\end{lemma}

\begin{proof}
Let us first prove~\eqref{eq:bnd_sup_nT}. Let $W_T \in \polP^{k+1}(T^1) \times \polP^{k+1}(T^2)$. Using the definition~\eqref{eq:def_RT} of the reconstruction operator, we infer that
\begin{align*}
n_T(P_T^{k+1}(U\upex),W_T) ={}& n_T(R_T^{k+1}(\hat J_T^{k+1}(U\upex)),W_T) \\
={}& n_T(J_T^{k+1}(U\upex),W_T) \\&- \sum_{\iot} \int_{\dTi} \kappa^i\GRAD w_T^i\SCAL\bn_T((J_T^{k+1}(U\upex))^i-\Pi^k_{\dTi}(u^i)) \\
={}& n_T(J_T^{k+1}(U\upex),W_T) \\&- \sum_{\iot} \int_{\dTi} \kappa^i\GRAD w_T^i\SCAL\bn_T((J_T^{k+1}(U\upex))^i-u^i),
\end{align*}
where we have exploited the choice for the face polynomials in the definition~\eqref{eq:def_hat_JT} of $\hat J_T^{k+1}(U\upex)$ and the fact that $\kappa^i\GRAD w_T^i\SCAL\bn_T\in \polP^k(\calF_{(\dT)^i})$.
Since $(U\upex_T)^i_{|\dTi}=u^i_{|\dTi}$, we obtain
\begin{align*}
n_T(P_T^{k+1}(U\upex)-U\upex_T,W_T) ={}&n_T(J_T^{k+1}(U\upex)-U\upex_T,W_T)
\\&- \sum_{\iot} \int_{\dTi} \kappa^i\GRAD w_T^i\SCAL\bn_T(J_T^{k+1}(U\upex)-U\upex_T)^i.
\end{align*}
To bound the first term on the right-hand side, we use the boundedness property~\eqref{eq:nT_bnd} of $n_T(\SCAL,\SCAL)$ from Lemma~\ref{lem:stab_bnd_nT} and $|\cdot|_{n_T\sharp}\le \|\cdot\|_{*T}$. To bound the second term, we use the Cauchy--Schwarz inequality followed by the discrete trace inequality~\eqref{eq:trace} to bound $\|\GRAD w_T^i\|_{\dTi}$. 

Let us now prove~\eqref{eq:bnd_Au-u_nT}. Let us set $Z_T=P_T^{k+1}(U\upex)-J_T^{k+1}(U\upex) \in \polP^{k+1}(T^1) \times \polP^{k+1}(T^2)$. Using the stability of $n_T$ from Lemma~\ref{lem:stab_bnd_nT}, we have
\begin{align*}
|Z_T|_{n_T}^2 &\lesssim n_T(Z_T,Z_T) \\
&= n_T(P_T^{k+1}(U\upex)-U\upex_T,Z_T)+n_T(U\upex_T-J_T^{k+1}(U\upex),Z_T).
\end{align*}
Using~\eqref{eq:bnd_sup_nT}, we can estimate the first term on the right-hand side as follows:
\[
n_T(P_T^{k+1}(U\upex)-U\upex_T,Z_T) \lesssim \|J_T^{k+1}(U\upex)-U\upex_T\|_{*T}  |Z_T|_{n_T}.
\]
Concerning the second term, we invoke the boundedness property~\eqref{eq:nT_bnd} of $n_T(\SCAL,\SCAL)$ from Lemma~\ref{lem:stab_bnd_nT} and $|\cdot|_{n_T\sharp}\le \|\cdot\|_{*T}$ to infer that
\begin{align*}
n_T(U\upex_T-J_T^{k+1}(U\upex),Z_T) &\lesssim |J_T^{k+1}(U\upex)-U\upex_T|_{n_T\sharp} |Z_T|_{n_T}\\
&\le \|J_T^{k+1}(U\upex)-U\upex_T\|_{*T} |Z_T|_{n_T}.
\end{align*}
Combining these two bounds, we infer that
\[
|Z_T|_{n_T} \lesssim \|J_T^{k+1}(U\upex)-U\upex_T\|_{*T}.
\]
Finally, using a triangle inequality leads to
\[
|P_T^{k+1}(U\upex)-U\upex_T|_{n_T} \le |Z_T|_{n_T} + |J_T^{k+1}(U\upex)-U\upex_T|_{n_T},
\]
which leads to the expected estimate since $|\cdot|_{n_T}\le \|\cdot\|_{*T}$.

Finally, let us prove~\eqref{eq:bnd_sT_hatJ}. We have
\[
s_T(\hat J_T^{k+1}(U\upex),\hat J_T^{k+1}(U\upex)) = \sum_{\iot} \kappa^ih_T^{-1} \|\Pi^k_{\dTi}((J_T^{k+1}(U\upex))^i-\Pi^k_{\dTi}(u^i))\|_{\dTi}^2,
\]
and observing that 
\begin{align*}
\|\Pi^k_{\dTi}((J_T^{k+1}(U\upex))^i-\Pi^k_{\dTi}(u^i))\|_{\dTi} &=
\|\Pi^k_{\dTi}((J_T^{k+1}(U\upex))^i-u^i)\|_{\dTi} \\
&= \|\Pi^k_{\dTi}((J_T^{k+1}(U\upex)-U\upex_T)^i)\|_{\dTi} \\
&\le \|(J_T^{k+1}(U\upex)-U\upex_T)^i\|_{\dTi},
\end{align*}
we infer the expected estimate. 
\end{proof}

\subsection{Consistency and boundedness}

We can now derive our key estimate regarding the consistency of the discrete problem~\eqref{eq:discrete_pb}. 

\begin{lemma}[Consistency and boundedness] \label{lem:consist}
Assume that $U\upex\in H^s(\Omega^1)\times H^s(\Omega^2)$, $s>\frac32$.
Let $\hat U_h\in \hat \calX_{h0}$ solve~\eqref{eq:discrete_pb}.
For all $\hat W_h \in \hat\calX_{h0}$, let us define
\begin{equation*}
\calF(\hat W_h) = \sum_{T\in\calT_h\uncut} \hat a_T\uncut(\hat \jmath_T^{k+1}(u\upex_T)-\hat u_T,\hat w_T)+ \sum_{T\in\calT_h^\Gamma} \hat a_T\cut(\hat J_T^{k+1}(U\upex)-\hat U_T,\hat W_T). 
\end{equation*}
Recall that $|\cdot|_{\hat a_h}$ is defined by~\eqref{eq:def_norm_ah}. 
The following holds true:
\begin{equation*}
|\calF(\hat W_h)| \lesssim \bigg( \sum_{T\in\calT_h\uncut} \|j_T^{k+1}(u\upex_T)-u\upex_T\|_{*T}^2+ \sum_{T\in\calT_h^\Gamma} \|J_T^{k+1}(U\upex)-U\upex_T\|_{*T}^2 \bigg)^{\frac12} |\hat W_h|_{\hat a_h}.
\end{equation*}
\end{lemma}

\begin{proof}
We first observe that, for all $T\in\calT_h\uncut$,
\begin{align*}
\hat a_T\uncut(\hat \jmath_T^{k+1}(u\upex_T),\hat w_T) = {}& a_T(p_T^{k+1}(u\upex_T),r_T^{k+1}(\hat w_T)) + s_T(\hat \jmath_T^{k+1}(u\upex_T),\hat w_T) \\
= {}&a_{T}(p_T^{k+1}(u\upex_T),w_T) +  s_T(\hat \jmath_T^{k+1}(u\upex_T),\hat w_T) \\
& - \int_{\dT}  \kappa_T \GRAD p_T^{k+1}(u\upex_T)\SCAL \bn_T(w_T-w_{\partial T}), 
\end{align*}
and for all $T\in\calT_h^\Gamma$,
\begin{align*}
\hat a_T\cut(\hat J_T^{k+1}(U\upex),\hat W_T) = {}& n_T(P_T^{k+1}(U\upex),R_T^{k+1}(\hat W_T)) + s_T(\hat J_T^{k+1}(U\upex),\hat W_T) \\
= {}&n_{T}(P_T^{k+1}(U\upex),W_T) +  s_T(\hat J_T^{k+1}(U\upex),\hat W_T) \\
& - \sum_{\iot} \int_{(\partial T)^i}  (\kappa \GRAD P_T^{k+1}(U\upex))^i\SCAL \bn_T(w_{T^i}-w_{(\dT)^i}), 
\end{align*}
where we have used the definitions~\eqref{eq:def_p_T} and~\eqref{eq:def_AT} of $p_T^{k+1}$ and $P^{k+1}_T$ and the definitions~\eqref{eq:def_rT} and~\eqref{eq:def_RT} of the reconstruction operators for $r^{k+1}_T(\hat w_T)$ and $R_T^{k+1}(\hat W_T)$ (and the symmetry of the bilinear forms $a_T$ and $n_T$). Moreover, using the fact that the discrete solution solves~\eqref{eq:discrete_pb}, we infer that
\[
\sum_{T\in\calT_h\uncut} \hat a_T\uncut(\hat u_T,\hat w_T) +
\sum_{T\in\calT_h^\Gamma} \hat a_T\cut(\hat U_T,\hat W_T) =: \Psi\uncut + \Psi\cut,
\]
where
\begin{align*}
\Psi\uncut =
\sum_{T\in\calT_h\uncut} \int_{T} fw_T = \sum_{T\in\calT_h\uncut} \bigg( \int_{T} \kappa_T \GRAD u\upex_T\SCAL \GRAD w_T - \int_{\partial T} \!\!\kappa_T\GRAD u\upex_T\SCAL \bn_T w_T \bigg),
\end{align*}
and
\begin{align*}
\Psi\cut &= 
\sum_{T\in\calT_h^\Gamma} \bigg( \sum_{\iot} \int_{T^i} fw_{T^i} + \int_{T^\Gamma} (g_{\textup{N}}w_{T^2}+g_{\textup{D}}\phi_T(W_T)) \bigg) \\
&= \sum_{T\in\calT_h^\Gamma} \bigg( \sum_{\iot} \int_{T^i} -\DIV(\kappa^i\GRAD u^i)w_{T^i} + \int_{T^\Gamma} (g_{\textup{N}}w_{T^2}+g_{\textup{D}}\phi_T(W_T)) \bigg) \\
&= \sum_{T\in\calT_h^\Gamma} \bigg(\sum_{\iot} \bigg( \int_{T^i} \kappa^i \GRAD u^i\SCAL \GRAD w_{T^i} - \!\!\int_{(\partial T)^i} \!\!(\kappa\GRAD u)^i\SCAL \bn_Tw_{T^i} \bigg) + n_{T\Gamma} (U\upex_T,W_T)\bigg),
\end{align*}
where we have used the following identity:
\begin{align*}
&- \sum_{\iot} \int_{T^\Gamma} \kappa^i\GRAD u^i\SCAL \bn_{T^i} w_{T^i}
 + \int_{T^\Gamma} (g_{\textup{N}}w_{T^2}+g_{\textup{D}}\phi_T(W_T)) \\
&= -\int_{T^\Gamma} \kappa^1\GRAD u^1\SCAL \bn_\Gamma \jump{W_T}_\Gamma + \int_{T^\Gamma} g_{\textup{D}}\phi_T(W_T)= n_{T\Gamma} (U\upex_T,W_T),
\end{align*}
recalling that $\jump{\kappa\GRAD u}_\Gamma\SCAL\bn_\Gamma=g_{\textup{N}}$ and $\jump{u}_\Gamma=g_{\textup{D}}$. Therefore, we have
\[
\Psi\cut = \sum_{T\in\calT_h^\Gamma} \bigg( n_T(U\upex_T,W_T) - \sum_{\iot} \int_{\dTi} (\kappa\GRAD U\upex_T)^i\SCAL \bn_T w_{T^i} \bigg).
\]
Putting the above identities together leads to $\calF(\hat W_h)=\calF\uncut(\hat W_h)+\calF^\Gamma(\hat W_h)$ with
\begin{align*}
\calF\uncut(\hat W_h) ={}& \sum_{T\in\calT_h\uncut} \bigg( a_{T}(p_T^{k+1}(u\upex_T)-u\upex_T,w_T)
+ s_T(\hat \jmath_T^{k+1}(u\upex_T),\hat w_T) \\
&- \int_{\dT}  \kappa_T \GRAD (p_T^{k+1}(u\upex_T)-u\upex_T)\SCAL \bn_T(w_{T}-w_{\dT}) \bigg), \\
\calF^\Gamma(\hat W_h) ={}& \sum_{T\in\calT_h^\Gamma} \bigg( n_{T}(P_T^{k+1}(U\upex)-U\upex_T,W_T)
+ s_T(\hat J_T^{k+1}(U\upex),\hat W_T) \\
&- \sum_{\iot} \int_{(\partial T)^i}  (\kappa \GRAD (P_T^{k+1}(U\upex)-U\upex_T))^i\SCAL \bn_T(w_{T^i}-w_{(\dT)^i}) \bigg).
\end{align*}
where we have used the continuity of the exact fluxes across $\dT$ for all $T\in\calT_h\uncut$ and across $(\dT)^i$ for all $\iot$ and $T\in\calT_h^\Gamma$ to add/subtract $w_{\dT}$ and $w_{(\dT)^i}$ in the integrals over $\dT$ and $\dTi$, respectively. 
It remains to bound the three terms composing $\calF\uncut(\hat W_h)$ and $\calF^\Gamma(\hat W_h)$ using Lemma~\ref{lem:approx_uncut} and Lemma~\ref{lem:prelim}, respectively. We only detail the bound on the three terms composing $\calF^\Gamma(\hat W_h)$ since the bound on $\calF\uncut(\hat W_h)$ uses similar arguments. To bound the first term, we use~\eqref{eq:bnd_sup_nT} and to bound the second term, we use~\eqref{eq:bnd_sT_hatJ}. For the third term, we use the Cauchy--Schwarz inequality so that we need to bound $(\kappa^i)^{\frac12} h_T^{\frac12} \|\GRAD (P_T^{k+1}(U\upex)-U\upex_T)^i\|_{\dTi}$ for all $\iot$. We can then add/subtract $(J_T^{k+1}(U\upex))^i$ and use the triangle inequality to obtain
\begin{align*}
(\kappa^i)^{\frac12} h_T^{\frac12}\|\GRAD (P_T^{k+1}(U\upex)-U\upex_T)^i\|_{\dTi}\le {}&(\kappa^i)^{\frac12} h_T^{\frac12}\|\GRAD (P_T^{k+1}(U\upex)-J_T^{k+1}(U\upex))^i\|_{\dTi} \\&+ (\kappa^i)^{\frac12} h_T^{\frac12}\|\GRAD (J_T^{k+1}(U\upex)-U\upex_T)^i\|_{\dTi}. 
\end{align*}
Since the second term on the right-hand side is bounded by $\|J_T^{k+1}(U\upex)-U\upex_T\|_{*T}$, we can focus on the first term. Using the discrete trace inequality~\eqref{eq:trace} followed by the triangle inequality where we add/subtract $(U\upex_T)^i$, we infer that
\begin{align*}
(\kappa^i)^{\frac12} h_T^{\frac12}\|\GRAD (P_T^{k+1}(U\upex)-J_T^{k+1}(U\upex))^i\|_{\dTi} \lesssim {}& (\kappa^i)^{\frac12} \|\GRAD (P_T^{k+1}(U\upex)-U\upex_T)^i\|_{T^i} \\
& + (\kappa^i)^{\frac12} \|\GRAD (U\upex_T-J_T^{k+1}(U\upex))^i\|_{T^i}.
\end{align*}
To conclude, we bound the first term using~\eqref{eq:bnd_Au-u_nT}, whereas the second term is readily bounded by $\|J_T^{k+1}(U\upex)-U\upex_T\|_{*T}$.
\end{proof}

\subsection{Main result}

We can now state our main result on the error analysis.

\begin{theorem}[Error estimate] \label{th:main}
Assume that $U\upex\in H^s(\Omega^1)\times H^s(\Omega^2)$, $s>\frac32$.
Let $\hat U_h\in \hat\calX_{h0}$ solve~\eqref{eq:discrete_pb}. Then, the following bound holds true:
\begin{align} \label{eq:err_est}
\calE := {}&\sum_{T\in\calT_h\uncut} \kappa_T\|\GRAD(u\upex_T-u_T)\|_T^2 +
\sum_{T\in\calT_h^\Gamma} \sum_{\iot} \kappa^i\|\GRAD (U\upex_T-U_T)^i\|_{T^i}^2 \nonumber \\
& + \sum_{T\in\calT_h^\Gamma} \kappa^1h_T^{-1}\|g_{\textup{D}}-\jump{U_T}_\Gamma\|_{T^\Gamma}^2 + \sum_{T\in\calT_h^\Gamma} (\kappa^2)^{-1}h_T\|g_{\textup{N}}-\jump{\kappa \GRAD U_T}_\Gamma\SCAL\bn_\Gamma\|_{T^\Gamma}^2 \nonumber \\
\lesssim {}&\sum_{T\in\calT_h\uncut} \|j_T^{k+1}(u\upex_T)-u\upex_T\|_{*T}^2
+ \sum_{T\in\calT_h^\Gamma} \|J_T^{k+1}(U\upex)-U\upex_T\|_{*T}^2 =:\calB.
\end{align}
Moreover, if $U\upex \in H^{k+2}(\Omega^1)\times H^{k+2}(\Omega^2)$, the following bounds hold true:
\begin{align} 
\calE &\lesssim \sum_{T\in\calT_h\uncut} \kappa_Th_T^{2(k+1)}|u\upex_T|_{H^{k+2}(T)}^2 +
\sum_{T\in\calT_h^\Gamma} \sum_{\iot} \kappa^i h_T^{2(k+1)}|E^i(u^i)|_{H^{k+2}(\Tdg)}^2 \nonumber \\
&\lesssim \sum_{\iot} \kappa^i h^{2(k+1)} |u^i|_{H^{k+2}(\Omega^i)}^2.\label{eq:cv_rate}
\end{align}
\end{theorem}

\begin{proof}
Let $\hat J_h \in \hat \calX_{h0}$ be such that its local components attached to the cells $T\in\calT_h\uncut$ are $\hat \jmath_T:=\hat \jmath^{k+1}_T(u\upex_T)$ and those attached to the cells $T\in\calT_h^\Gamma$ are $\hat J_T:=\hat J^{k+1}_T(U\upex)$ (the face components of $\hat J^h$ are indeed well defined). Using stability (Lemma~\ref{lem:stab_aT} and~\eqref{eq:def_norm_ah}), consistency/boundedness (Lemma~\ref{lem:consist}), and the Cauchy--Schwarz inequality, we infer that
\begin{align*}
|\hat J_h-\hat U_h|_{\hat a_h}^2 &= 
\sum_{T\in\calT_h\uncut} |\hat \jmath_T-\hat u_T|_{\hat a_T}^2 +
\sum_{T\in\calT_h^\Gamma} |\hat J_T-\hat U_T|_{\hat a_T}^2 \\
& \lesssim \sum_{T\in\calT_h\uncut} \hat a_T\uncut(\hat \jmath_T-\hat u_T,\hat \jmath_T-\hat u_T) + 
\sum_{T\in\calT_h^\Gamma} \hat a_T\cut(\hat J_T-\hat U_T,\hat J_T-\hat U_T) \\
&= \calF(\hat J_h-\hat U_h) \lesssim \calB^{\frac12} |\hat J_h-\hat U_h|_{\hat a_h}.
\end{align*}
This implies that
\[
\sum_{T\in\calT_h\uncut} |\hat \jmath_T-\hat u_T|_{\hat a_T}^2 +
\sum_{T\in\calT_h^\Gamma} |\hat J_T-\hat U_T|_{\hat a_T}^2 \lesssim \calB.
\]
Recalling the definitions~\eqref{eq:def_seminorm_aT_uncut} and~\eqref{eq:def_seminorm_aT} of $|\cdot|_{\hat a_T}$, we infer that
\begin{equation} \label{eq:bnd_err_disc1}
\sum_{T\in\calT_h\uncut} |j^{k+1}_T(u\upex_T)-u_T|_{a_T}^2 + 
\sum_{T\in\calT_h^\Gamma} |J^{k+1}_T(U\upex)-U_T|_{n_T}^2 \lesssim \calB,
\end{equation} 
and using the discrete trace inequality~\eqref{eq:trace}, we also infer that
\begin{equation} \label{eq:bnd_err_disc2}
\sum_{\iot} \kappa^ih_T\|\GRAD(J^{k+1}_T(U\upex)-U_T)^i\|_{T^\Gamma}^2 \lesssim \calB.
\end{equation} 
For all $T\in\calT_h\uncut$, we add/subtract $u\upex_T$ in~\eqref{eq:bnd_err_disc1} and we use that
$|\cdot|_{a_T}\le \|\cdot\|_{*T}$; for all $T\in\calT_h\cut$, we add/subtract
$U\upex_T$ in \eqref{eq:bnd_err_disc1}-\eqref{eq:bnd_err_disc2} and we use that $|\cdot|_{n_T}^2 + \sum_{\iot} \kappa^ih_T\|\GRAD(\cdot)\|_{T^\Gamma}^2 \le\|\cdot\|_{*T}^2$, $\jump{U\upex_T}_\Gamma=g_{\textup{D}}$, and
\[
(\kappa^2)^{-\frac12}h_T^{\frac12}\|g_{\textup{N}}-\jump{\kappa \GRAD U_T}_\Gamma\SCAL\bn_\Gamma\|_{T^\Gamma} \le \sum_{\iot} (\kappa^ih_T)^{\frac12} \|\GRAD(U\upex_T-U_T)^i\|_{T^\Gamma},
\]
since $\jump{\kappa \GRAD U\upex_T}_\Gamma\SCAL\bn_\Gamma=g_{\textup{N}}$ and $\kappa^1 < \kappa^2$. This leads to~\eqref{eq:err_est}. Finally, the estimate~\eqref{eq:cv_rate} follows by combining~\eqref{eq:err_est} with Lemmas~\ref{lem:app_j_uncut} and~\ref{lem:app}.
\end{proof}



\section{Building the mesh}
\label{sec:build_mesh}

In this section, we show how to build a mesh satisfying Assumption~\ref{ass:Gamma} and Assumption~\ref{ass:ball}. Our goal is not to propose an optimized construction, but simply to show that both Assumptions can be satisfied. \cor{A more practically-oriented discussion on algorithmic aspects is postponed to future work.} We assume that we are initially given a shape-regular (polyhedral) mesh. Our goal is to show that we can satisfy Assumption~\ref{ass:Gamma} by refining the mesh and then Assumption~\ref{ass:ball} by means of a local cell-agglomeration procedure. 

The shape-regularity of the mesh sequence implies that 
there is $\rho\in(0,1)$ so that the following geometric properties hold true for all $T\in\calT_h$, \textup{(i)} there is $\bx_T\in T$ so that $B(\bx_T,\rho h_T) \subset T$ ; \textup{(ii)} $T_\rho := \{\bx\in \Real^d,\; d(\bx,T)\le \rho h_T\} \subset \Delta(T)$ where $\Delta(T) := \{T'\in\calT_h\tq T\cap T'\ne\emptyset\}$ is the collection of the mesh cells touching $T$; 
\textup{(iii)} $\rho \max_{T'\in \Delta(T)}h_{T'} \le \min_{T''\in\Delta(T)}h_{T''}$; 
\textup{(iv)} for all $\bx\in\Omega$ with $d(\bx,\partial\Omega)\ge h$, and all $\alpha\in (0,1)$, letting $T_1\in\calT_h$ be s.t.~$\bx\in T_1$, there is $T_2 \in\calT_h$ s.t.~$T_2\cap B(\bx,\alpha h_{T_1})$ has positive $d$-measure and there is $\bx_{T_2}\in T_2$ so that $B(\bx_{T_2},\rho\alpha h_{T_2})\subset T_2 \cap B(\bx,\alpha h_{T_1})$; this last property means that for any open ball (not too close to the boundary), there is at least one mesh cell s.t.~its intersection with this ball contains a smaller ball with equivalent diameter.

\subsection{Assumption~\ref{ass:Gamma}: mesh refinement}\label{sec:Gamma_ref}

Let us show that Assumption~\ref{ass:Gamma} can be satisfied if the mesh is fine enough.

\begin{lemma}[Assumption~\ref{ass:Gamma}]
Assumption~\ref{ass:Gamma} holds true with $\gamma=\frac14$ provided $hM\le 1$ where $M$ is an upper bound on the curvature of $\Gamma$.
\end{lemma} 

\begin{proof}
Let $T\in\calT_h^\Gamma$. Fix a point $\bs_0 \in T^\Gamma$ and introduce the local coordinates $\bxi = (\bxi',\xi_d)$, with zero at $\bs_0$, where $\bxi' \in \mathbb{R}^{d-1}$ are the coordinates in the tangent plane $T_{\bs_0}\Gamma$ and $\xi_d$ is the coordinate in the normal direction to the tangent plane at $\bs_0$. Owing to the assumption $hM\le 1$, we can write $T^\Gamma=\{\bs := (\bxi',\psi(\bxi')), \; \bxi'\in V(\bzero)\}$ where $V(\bzero)$ is a neighborhood of $\bzero$ in $\Real^{d-1}$ and $\psi : V(\bzero)\to \Real$ is a smooth map. Note that $\psi(\bzero)=0$, $\nabla_{\bxi'}\psi(\bzero)=\bzero$, and that a normal vector to the tangent plane $T_\bs\Gamma$ is $\bn(\bxi')=(-\nabla_{\bxi'}\psi(\bxi'),1)$. Let us set $\hat \bx_T = (\bzero,-2h_T)$ and consider the function
\[
f(\bxi') = (\bs-\hat\bx_T) \cdot \bn_{\Gamma}(\bxi') = -\bxi' \cdot \nabla_{\bxi'} \psi(\bxi') + \psi(\bxi') + 2 h_T.
\]
Then $f(\bzero) = 2 h_T$, and since $\nabla_{\bxi'} f(\bxi') = -\bxi' \cdot D_{\bxi'\bxi'}^2 \psi(\bxi')$ and $\|\bxi'\|_{\ell^2}\le h_T$, we infer that 
\[
f(\bxi') \ge f(\bzero) - h_T\|\nabla_{\bxi'} f\|_{L^\infty(V(\bzero))} 
\ge 2h_T-h_T^2M \ge h_T.
\]
Since $\|\bn(\bxi')\|_{\ell^2} \le 1 + \|\nabla_{\bxi'}\psi(\bxi')\|_{\ell^2} \le 1+hM\le 2$, we infer that
\[
d(\hat \bx_T,T_\bs\Gamma) = \|\bn(\bxi')\|_{\ell^2}^{-1} f(\bxi') 
\ge \frac12 h_T, \qquad \forall \bs\in T^\Gamma.
\]
In addition, we have
$\|\hat\bx_T-\bs\|_{\ell^2} \le \|\bxi'\|_{\ell^2} + |\psi(\bxi')+2h_T|
\le 3h_T+|\psi(\bxi')| \le 4h_T$ since $\psi(\bzero)=0$, $\nabla_{\bxi'}\psi(\bzero)=\bzero$ and $hM\le 1$.
\end{proof}

\subsection{Assumption~\ref{ass:ball}: local cell-agglomeration} \label{sec:agglo}

Assume that we are given an initial shape-regular (polyhedral) mesh $\calT_h$ (with parameter $\rho$) that satisfies Assumption~\ref{ass:Gamma} (with parameter $\gamma$), but that does not satisfy Assumption~\ref{ass:ball}. We now describe a simple local cell-agglomeration procedure to produce a new mesh that is still shape-regular and that satisfies Assumption~\ref{ass:Gamma} and Assumption~\ref{ass:ball}. The main idea is that we eliminate any mesh cell in $\calT_h$ that is cut unfavorably by the interface by merging this cell with a neighboring one. 
\cor{An illustration is provided in Figure~\ref{fig:agglo}.}
\begin{figure}[htb]
\begin{center}
\includegraphics[width=0.6\linewidth]{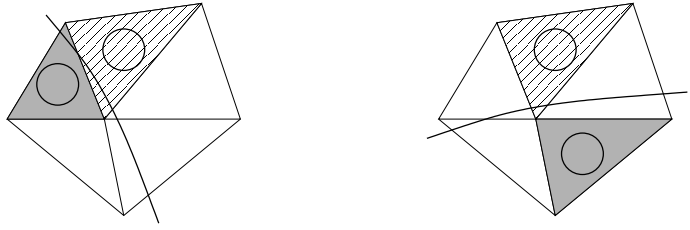}
\end{center}
\caption{\cor{Triangular mesh cell $T$ (filled by dashes) cut unfavorably by the interface; the companion cell in the agglomeration procedure is shown in gray for two generic situations: on the left, both cells share a face, and on the right, they only share a vertex; in both cases, the companion cell is in $\Delta(T)$, and the two balls in the agglomerated cell for Assumption~\ref{ass:ball} are shown.}}
\label{fig:agglo}
\end{figure}
We consider the partition $\calT_h = \calT_h^1\cup \calT_h^\Gamma\cup \calT_h^2$, and picking a value $\delta\in(0,1)$ (the precise value of $\delta$ is determined below), we further partition $\calT_h^\Gamma$ into 
\begin{equation} \label{eq:partition_T_Gamma}
\calT_h^\Gamma = \calT_h^{\textsc{ok}} \cup \calT_h^{\textsc{ko},1}\cup\calT_h^{\textsc{ko},2},
\end{equation} 
where $T\in \calT_h^{\textsc{ok}}$ iff the condition~\eqref{eq:boule_in_Ti} from Assumption~\ref{ass:ball} holds true for all $i\in\{1,2\}$, whereas $T\in \calT_h^{\textsc{ko},i}$ if the condition fails for $i\in\{1,2\}$. Let us first give two useful lemmas underpinning our local cell-agglomeration procedure. 

\begin{lemma}[Partition of $\calT_h^\Gamma$] \label{lem:disjoint}
The subsets $\calT_h^{\textsc{ko},1}$ and $\calT_h^{\textsc{ko},2}$ are disjoint if the mesh-size $h$ is small enough and if $\delta \le \frac13\rho$.
\end{lemma}

\begin{proof}
For $\bs\in\Gamma$ and positive real numbers $\alpha,\beta$, we define the strip of length $\alpha$ and aspect ratio $\beta$ centered at the point $\bs$ and aligned with the tangent plane $T_\bs\Gamma$ as
\[
S_\Gamma(\bs,\alpha,\beta):=\bs + \{\bt+\bn,\; \bt\in T_\bs\Gamma,\; 2\|\bt\|_{\ell^2}\le \alpha, \; \bn \in (T_\bs\Gamma)^\perp,\; 2\|\bn\|_{\ell^2}\le \alpha\beta
\}.
\]
The regularity of $\Gamma$ implies that, for all $\lambda\in (0,1]$, there is $\delta(\lambda)>0$, so that, for all $\bs\in\Gamma$ and all $\alpha\in (0,\delta(\lambda)]$, 
\begin{equation} \label{eq:strip}
\Gamma \cap B(\bs,\alpha) \subset S_\Gamma(\bs,2\alpha,\lambda).
\end{equation}
(Note that the diameter of $B(\bs,\alpha)$ is $2\alpha$.)
Let $T\in\calT_h^\Gamma$ and let $\bs \in T^\Gamma$. Recall from mesh regularity (property~(i)) that $B(\bx_T,\rho h_T) \subset T$. Let us apply~\eqref{eq:strip} with $\lambda=\frac13\rho$. Assume that $h\le \delta(\frac13\rho)$. Then $T^\Gamma \subset \Gamma \cap B(\bs,h_T) \subset S_\Gamma(\bs,2h_T,\frac13\rho)$. Elementary geometric considerations show that there is a point $\tilde\by_T \in T$ so that $B(\tilde\by_T,\frac13\rho h_T) \subset B(\bx_T,\rho h_T)\setminus S_\Gamma(\bs,2h_T,\frac13\rho)$, which implies, in particular, that $B(\tilde\by_T,\frac13\rho h_T) \cap \Gamma = \emptyset$. Therefore, $B(\tilde\by_T,\frac13\rho h_T)$ is a subset of either $T^1$ or $T^2$. 
\end{proof}

\begin{lemma}[Finding a suitable neighbor] \label{lem:neighbor}
Assume that the mesh-size is small enough and take $\delta=\frac14 \rho^3$. 
Let $\iot$. 
For all $T\in \calT_h^{\textsc{ko},i}$, there is a mesh cell in $\Delta(T)$ such that the condition~\eqref{eq:boule_in_Ti} holds true for $i$, i.e., this mesh cell is in $(\calT_h^i \cup \calT_h^{\textsc{ok}} \cup \calT_h^{\textsc{ko},\overline\i}) \cap \Delta(T)$, where $\overline\i=3-i$ (so that $\overline\i=2$ if $i=1$ and $\overline\i=1$ if $i=2$).
\end{lemma}

\begin{proof}
Fix $\iot$ and let $T\in \calT_h^{\textsc{ko},i}$. 
Owing to mesh regularity (property~(ii)), we have $T_\rho := \{\bx\in \Real^d,\; d(\bx,T)\le \rho h_T\} \subset \Delta(T)$. Let $\bs \in T^\Gamma$. Assume that $h\le \delta(\frac14\rho)$ (see~\eqref{eq:strip}), so that 
$\Gamma \cap B(\bs,h_T) \subset S_\Gamma(\bs,2h_T,\frac14\rho)$. Since the width of $S_\Gamma$ is smaller than or equal to $\tfrac12\rho h_T$, there is a ball $B(\bs',\frac14\rho h_T) \subset T_\rho \cap\Omega^i \setminus S_\Gamma(\bs,2h_T,\frac14\rho)$. Note that $d(\bs',\partial\Omega)\ge d(\bs,\partial \Omega)-h\ge h$ since $d(\Gamma,\partial\Omega)\ge2h$ and $\bs\in\Gamma$. Since $\bs'\in T_\rho$, there is $T_1\in\Delta(T)$ s.t.~$\bs'\in T_1$. Using mesh regularity (property (iv) with $\alpha = \frac14\rho h_T h_{T_1}^{-1}\le \frac14$ owing to property~(iii)), we infer that there is $T_2 \in \calT_h$ so that $T_2\cap B(\bs',\frac14 \rho h_{T})$ has positive $d$-measure and there is a ball $B(\bx_{T_2},\rho\alpha h_{T_2})\subset T_2 \cap B(\bs',\frac14 \rho h_{T})$. Since $\rho\alpha=\frac14\rho^2h_T h_{T_1}^{-1} \ge \frac14 \rho^3 = \delta$ (using again property (iii)), we infer that the mesh cell $T_2$ satisfies the condition~\eqref{eq:boule_in_Ti} for $i$. Moreover, $T_2\cap T_\rho$ has positive $d$-dimensional measure, so that $T_2\in \Delta(T)$. This concludes the proof. 
\end{proof}

We can now present our local cell-agglomeration procedure. We consider the partition~\eqref{eq:partition_T_Gamma} with $\delta:=\frac14\rho^3$, and assume that the mesh-size is small enough so that Lemma~\ref{lem:disjoint} and Lemma~\ref{lem:neighbor} hold true (note that $\delta\le \frac13\rho$). The procedure is as follows:
(1) For all $T\in \calT_h^{\textsc{ko},1}$, we choose a neighboring mesh cell $N_1(T) \in (\calT_h^{\textsc{ok}} \cup \calT_h^1\cup \calT_h^{\textsc{ko},2}) \cap \Delta(T)$ (this is possible owing to Lemma~\ref{lem:neighbor}). We denote the collection of the cells in $\calT_h^{\textsc{ko},2}$ chosen in the above step as the subset $\hat \calT_h^{\textsc{ko},2}\subset \calT_h^{\textsc{ko},2}$. (2) For all $T\in \calT_h^{\textsc{ko},2} \setminus \hat \calT_h^{\textsc{ko},2}$, we choose a neighboring mesh cell $N_2(T) \in (\calT_h^{\textsc{ok}} \cup \calT_h^2\cup \calT_h^{\textsc{ko},1}) \cap \Delta(T)$.
(3) For all $i\in\{1,2\}$, let $\calN_i$ be the collection of all the cells in $\calT_h^{\textsc{ok}} \cup \calT_h^i \cup \calT_h^{\textsc{ko},\overline\i}$ that have been selected at least once in one of the two previous steps. For all $T^\sharp\in \calN_1\cup \calN_2$, we define the agglomerated cell 
\begin{equation} \label{eq:star_sharp}
T^* := T^\sharp\cup \{T\in \calT_h^{\textsc{ko},1},\; N_1(T)=T^\sharp\} \cup \{T\in \calT_h^{\textsc{ko},2} \setminus \hat \calT_h^{\textsc{ko},2}, \; N_2(T) = T^\sharp\},
\end{equation}
and we observe that $T^* \subset \Delta(T^\sharp)$.
We collect all the agglomerated cells in $\calT_h^{\textup{agglo}}$, and we define the new mesh
\begin{equation}
\calT_h^* := \bigg( (\calT_h^{\textsc{ok}} \cup \calT_h^1 \cup \calT_h^2)\setminus (\calN_1\cup \calN_2)\bigg) \; \cup\; \calT_h^{\textup{agglo}}.
\end{equation}

\begin{remark}[Choice of $N_i(T)$]
In Steps 1 and 2, we do not require that $T$ and $N_i(T)$ share a face, it is sufficient that they share a point (actually, it is just sufficient that the set $T\cup N_i(T)$ has a diameter of order $h_T$, but we do not explore this further here). Thus, there is some freedom in the choice of $N_i(T)$. In practice, one can choose $N_i(T)$ sharing a face with $T$ whenever possible.
\end{remark}

It is easy to see that the newly created mesh $\calT_h^*$ is still shape-regular and satisfies Assumption~\ref{ass:Gamma}. Shape regularity follows since the agglomeration of a finite number of shape-regular neighbors remains shape regular, but with a possibly smaller parameter $\rho^*<\rho$. Assumption~\ref{ass:Gamma} is satisfied since each cell in the original mesh satisfies the assumption and any finite union of cells satisfying this assumption must also satisfy it, but once again with a possibly smaller parameter $\gamma^*<\gamma$.
Let us finally verify that $\calT_h^*$ also satisfies Assumption~\ref{ass:ball}. 

\begin{lemma}[Assumption~\ref{ass:ball}] \label{lemma:boule_in_T}
Assume that the mesh-size is small enough and that the cell-agglomeration procedure uses the cut parameter $\delta=\frac14\rho^3$. Then 
Assumption~\textup{\ref{ass:ball}} holds true for the mesh $\calT_h^*$ with $\delta^*=\frac13\rho \delta=\frac{1}{12}\rho^4$.
\end{lemma}

\begin{proof}
Let $T^*\in\calT_h^*$ be s.t.~$\mes_{d-1}(T\cap\Gamma)>0$. Then, 
$T^* \in \calT_h^{\textsc{ok}} \setminus (\calN_1\cup\calN_2)$ or $T^*\in \calT_h^{\textup{agglo}}$. In the first case, $T^*$ is also a mesh cell from the original mesh $\calT_h$, and the definition of $\calT_h^{\textsc{ok}}$ implies that the condition~\eqref{eq:boule_in_Ti} is satisfied with the cut parameter $\delta$, and therefore also with the cut parameter $\delta^*\le \delta$. In the second case where $T^*\in \calT_h^{\textup{agglo}}$, let us assume to fix the ideas that the associated cell $T^\sharp$ (see~\eqref{eq:star_sharp}) is in $\calN_1$, so that $T^\sharp=N_1(T_0^\sharp)$ with $T_0^\sharp \in \calT_h^{\textsc{ko},1}$. Owing to Lemma~\ref{lem:disjoint}, the condition~\eqref{eq:boule_in_Ti} is satisfied in $T^\sharp_0$ with parameter $\delta$ and $i=2$, and by construction, this condition is satisfied in $T^\sharp$ with parameter $\delta$ and $i=1$. Since $h_{T^*}\le h_{\Delta(T^\sharp)} \le 3\max_{T'\in\Delta(T^\sharp)} h_{T'} \le
3\rho^{-1} \min(h_{T^\sharp},h_{T_0^\sharp})$ owing to mesh regularity (property (iii)), we infer that the condition~\eqref{eq:boule_in_Ti} is satisfied in $T^*$ with parameter $\delta^*$ and all $\iot$.
\end{proof}

\bibliographystyle{abbrvnat} 
\bibliography{biblio}

\end{document}

%% file: cutcell1_gd.pstex_t
\begin{picture}(0,0)%
\includegraphics{cutcell1_gd.pstex}%
\end{picture}%
\setlength{\unitlength}{4144sp}%
\begingroup\makeatletter\ifx\SetFigFont\undefined%
\gdef\SetFigFont#1#2#3#4#5{%
  \reset@font\fontsize{#1}{#2pt}%
  \fontfamily{#3}\fontseries{#4}\fontshape{#5}%
  \selectfont}%
\fi\endgroup%
\begin{picture}(9615,2049)(2146,-3223)
\put(10711,-2311){\makebox(0,0)[lb]{\smash{{\SetFigFont{17}{20.4}{\familydefault}{\mddefault}{\updefault}{\color[rgb]{0,0,0}$\tilde\bx_{T^1}$}%
}}}}
\put(4951,-3076){\makebox(0,0)[lb]{\smash{{\SetFigFont{17}{20.4}{\familydefault}{\mddefault}{\updefault}{\color[rgb]{0,0,0}$\hat\bx_T$}%
}}}}
\put(6751,-1996){\makebox(0,0)[lb]{\smash{{\SetFigFont{17}{20.4}{\familydefault}{\mddefault}{\updefault}{\color[rgb]{0,0,0}$T_{\bs}\Gamma$}%
}}}}
\put(10711,-1321){\makebox(0,0)[lb]{\smash{{\SetFigFont{17}{20.4}{\familydefault}{\mddefault}{\updefault}{\color[rgb]{0,0,0}$\tilde\bx_{T^2}$}%
}}}}
\put(9541,-2806){\makebox(0,0)[lb]{\smash{{\SetFigFont{17}{20.4}{\familydefault}{\mddefault}{\updefault}{\color[rgb]{0,0,0}$\delta h_T$}%
}}}}
\put(9226,-1951){\makebox(0,0)[lb]{\smash{{\SetFigFont{17}{20.4}{\familydefault}{\mddefault}{\updefault}{\color[rgb]{0,0,0}$\delta h_T$}%
}}}}
\put(5266,-1996){\makebox(0,0)[lb]{\smash{{\SetFigFont{17}{20.4}{\familydefault}{\mddefault}{\updefault}{\color[rgb]{0,0,0}$\bs$}%
}}}}
\put(6526,-2941){\makebox(0,0)[lb]{\smash{{\SetFigFont{17}{20.4}{\familydefault}{\mddefault}{\updefault}{\color[rgb]{0,0,0}$\|\hat\bx_T-\bs\|_{\ell^2}\le\gamma^{-1}h_T$}%
}}}}
\put(2161,-2941){\makebox(0,0)[lb]{\smash{{\SetFigFont{17}{20.4}{\familydefault}{\mddefault}{\updefault}{\color[rgb]{0,0,0}$d(\hat\bx_T,T_{\bs}\Gamma)\ge \gamma h_T$}%
}}}}
\put(6751,-1366){\makebox(0,0)[lb]{\smash{{\SetFigFont{17}{20.4}{\familydefault}{\mddefault}{\updefault}{\color[rgb]{0,0,0}$\Gamma$}%
}}}}
\put(11746,-1411){\makebox(0,0)[lb]{\smash{{\SetFigFont{17}{20.4}{\familydefault}{\mddefault}{\updefault}{\color[rgb]{0,0,0}$\Gamma$}%
}}}}
\end{picture}%